\documentclass[a4paper,11pt]{amsart}
\usepackage[T1]{fontenc}
\usepackage[latin9]{inputenc}
\usepackage{geometry}
\usepackage[english]{babel}
\usepackage{amstext}
\usepackage{amsthm}
\usepackage{mathtools}
\usepackage{amssymb}
\usepackage{xcolor}
\usepackage[unicode=true,
 bookmarks=true,bookmarksnumbered=false,bookmarksopen=false,
 breaklinks=false,pdfborder={0 0 1},backref=false,colorlinks=false]
 {hyperref}
 \usepackage{graphicx}
 \usepackage{caption}
 \usepackage{float}
\usepackage[caption = false]{subfig}

\allowdisplaybreaks

\makeatletter
\theoremstyle{plain}
\newtheorem{thm}{Theorem}
\theoremstyle{plain}

\theoremstyle{plain}
\newtheorem{lem}{Lemma}
\newtheorem{cor}{Corollary} 

\makeatother

\newcommand{\p}{\mathbb{P}}

\begin{document}
\title[Asymptotic behavior of the number variance of random samples]{On the asymptotic behavior of the number variance of random points on the unit torus}
\author{Christoph Aistleitner}
\address{TU Graz, Graz, Austria}
\email{aistleitner@math.tugraz.at}
\author{Siegfried H\"ormann}
\address{TU Graz, Graz, Austria}
\email{shoermann@tugraz.at}
\author{Maryna Manskova}
\address{TU Graz, Graz, Austria}
\email{maryna.manskova@tugraz.at}

\begin{abstract} 
The study of the distribution of real sequences in the unit torus has a long and rich history. There are many notions which describe the ``pseudo-random'' behavior of a given deterministic sequence (often of arithmetic origin), such as equidistribution (which describes the distribution on a ``global scale'') or the pair correlation (which is a ``local-scale'' statistics). Another popular and important statistics, which can be applied to all ranges from global via intermediate to local, is the number variance $V_N(s)$, which describes the mean-square fluctuation of the number of elements of an initial segment of the sequence in an interval of length $s$, when moving this interval around the torus. In the terminology of quantum chaology, where such statistics are used to describe the distribution of the energy levels of a system, pseudo-random behavior is called Poissonian behavior. Since the number variance is an important indicator to qualify pseudo-random behavior, it is quite surprising that the literature does not seem to contain a precise description of the distribution and of the almost sure asymptotic behavior of the number variance in the truly random case (i.i.d.\ random points), even in the comparatively simple setup of uniformly distributed points on the unit torus. In the present paper we calculate the distribution of the number variance with high precision throughout a wide range of the parameter $s$, and determine the precise threshold (in terms of $s$) where the ``Poissonian'' asymptotics $V_N(s) \sim N s$ breaks for a typical realization of a random sequence. Key technical input comes from the Koml\'{o}s--Major--Tusn\'{a}dy theorem and the Karhunen-Lo\`eve decomposition. 
\end{abstract}

\maketitle

\begin{center}
    \itshape Dedicated to the memory of Vera T. S\'os (1930-2023).
\end{center}

\section{Introduction and statement of results}
Let $x_1, x_2, \dots$ be a sequence of points in $[0,1]$. For $x \in (0,1)$ and $s \in (0,1)$ we define the counting function 
\begin{equation} \label{S_def}
S_N (x,s) = \# \left\{ 1 \leq i \leq N:~ x_i \in (x, x+s] \right\}.
\end{equation}
Here, and throughout the rest of this paper, we understand all intervals as sub-intervals of the unit torus. In other words, when $x+s \leq 1$ then $(x, x+s]$ is just an ordinary interval, but when $x+s>1$ then we set
$$
(x,x+s] = (x,1] \cup [0,\{x+s\}],
$$
where $\{ \cdot \}$ denotes the fractional part function. Later in the paper, $s$ will be treated as a parameter which approaches zero as $N \to \infty$. Clearly, if the points $x_1, x_2, \dots$ are random, meaning independent, identically distributed (i.i.d.) points that are sampled from a uniform distribution on $[0,1]$, then one can easily calculate that $\mathbb{E} (S_N(x,s)) = Ns$, and accordingly one would expect $S_N(x,s)$ to be typically of asymptotic order $N s$ as $N \to \infty$. The mean-square fluctuations of $S_N(x,s)$ as $x$ moves through $[0,1]$ are described by the number variance of $x_1, \dots, x_N$, which is defined as
\begin{equation} \label{V_def}
V_N(s) = \int_{0}^{1} S_N(x,s)^2  - N^2 s^2 ~dx = \int_{0}^{1} \left(S_N(x,s)  - Ns \right)^2 ~dx.
\end{equation}
The typical order of $V_N$ in the truly random setting is more difficult to see heuristically than that of $S_N$, but one can show with a bit of calculation that for such random points one has $\mathbb{E} \left( V_N(s) \right) = Ns(1-s)$, which means that one would typically expect $V_N(s)$ to be of size roughly $Ns$. In this setup pseudo-random behavior is called ``Poissonian'' behavior, so that a deterministic sequence is said to exhibit Poissonian behavior (with respect to the number variance) if its number variance satisfies $V_N(s) \sim N s$ as $N \to \infty$ (here, and in the sequel, we write $f(x) \sim g(x)$ if $\frac{f(x)}{g(x)} \to 1$ as $x \to \infty$). However, this ``definition'' of Poissonian behavior is actually rather vague, since the number variance is supposed to describe the pseudo-random behavior of a point set on different scales, ranging from highly localized to more global scaling (but usually assuming that $s \to 0$ as $N \to \infty$). Thus one has to clarify in which sense the asymptotic relation $V_N(s) \sim N s$ has to be understood when one simultaneously assumes that $N \to \infty$ and $s \to 0$. It is natural to demand that $V_N(s) \sim N s$ should hold uniformly throughout a large range of possible values of $s$. The key purpose of this paper is to determine precisely what the largest possible range of such values of $s$ is, in the case of a sequence of random points. More precisely, our aim is to determine the maximal size of $(s_N)_{N \geq 1}$ such that $V_N(s) \sim s N$ is true almost surely, uniformly for all $s \in [0, s_N]$, while this asymptotics will fail to be true almost surely for larger values of $s$. It will turn out that this critical size of $s_N$ is a bit below the order of $1/\log \log N$. This is relevant for the characterization of Poissonian behavior of deterministic sequences, since the asymptotic behavior of the number variance of a deterministic sequence should clearly be compared with the asymptotic behavior of a typical (almost sure) realization of a random sequence.\\

The number variance has long been a central object of study in theoretical physics and mathematical physics, particularly within the framework of Quantum Chaos. According to the Bohigas--Giannoni--Schmit (BGS) conjecture, the statistical properties of the energy levels of a quantum system are determined by the dynamics of the corresponding classical system. More specifically, the Berry--Tabor conjecture asserts that the spectral statistics of systems whose classical counterparts are integrable are expected to follow those of a Poisson point process. In this context, the number variance $V_N(s)$ is used to measure the rigidity of the spectrum (see, e.g., the classic treatments by Mehta \cite{mehta} and Haake, Gnutzmann and Ku\'{s} \cite{hgk}). For a Poissonian spectrum, the variance grows linearly with the mean ($V_N(s) \sim Ns$), indicating a lack of correlation between energy levels -- they are free to cluster or gap. In contrast, chaotic systems (such as those modeled by Gaussian ensembles in Random Matrix Theory) exhibit spectral rigidity, where the number variance grows only logarithmically: $V_N(s) \sim c\log(Ns)$, for some suitable constant $c>0$. This reduced variance is a hallmark of so-called level repulsion. We note in passing that a profound connection to number theory arises here through the work of Montgomery \cite{montgomery}, who conjectured that the pair correlation of the non-trivial zeros of the Riemann zeta function follows the GUE (Gaussian Unitary Ensemble) model of Random Matrix Theory. This implies that the zeros of $\zeta(s)$ conjecturally exhibit the same spectral rigidity as the energy levels of a chaotic quantum system; for a comprehensive treatment of this correspondence and its generalizations to other $L$-functions, we refer to Katz and Sarnak \cite{ks_book} and Conrey \cite{conrey}. Since $V_N(s)$ is of a very different order in the integrable resp.\ chaotic setting, the number variance is one of the primary tools used to distinguish these two physical regimes, and serves as a standard benchmark for pseudo-randomness. While these phenomena seem to be quite well-understood at the level of conjectures and expectations, rigorous mathematical results for specific systems are relatively scarce and usually very hard to obtain (among the most important ``mathe\-matical'' papers establishing Poissonian local statistics for specific integrable systems are those of Eskin, Margulis and Mozes \cite{emm}, Marklof \cite{marklof}, Sarnak \cite{sarnak} and VanderKam \cite{vk1,vk2}). It should also be noted that when it comes to the number variance, in large parts of the physics literature the Poissonian baseline $V_N(s) \sim Ns$ seems to be treated in a vague sense as an asymptotic limit in expectation or in distribution, without a rigorous treatment of the precise range of $s$ for which this asymptotic behavior holds or of the (almost sure) fluctuations for a single, typical realization of the system.\\

From a number-theoretic perspective, the study of pseudo-randomness properties such as the number variance and the pair correlation of sequences in $[0,1]$ can be seen as a refinement of the more classical theory of uniform distribution modulo one. Recall that a sequence $(x_n)_{n \geq 1}$ in $[0,1]$ is called uniformly distributed mod 1 (u.d.\ mod 1) if, using the notation from the initial paragraph, 
\begin{equation} \label{ud_def}
\lim_{N \to \infty} \frac{1}{N} S_N(x,s) = s \qquad \text{for all $x$ and $s$}.
\end{equation}
It is important to point out that in contrast to the local statistics discussed before, the test interval $(x,x+s]$ in the u.d.\ mod 1 framework \eqref{ud_def} is understood to remain permanently fixed as $N \to \infty$. There is a quantitative form of the uniform distribution concept called discrepancy, and in the language of this theory the number variance is a form of the $L^2$-discrepancy (see \cite{drti,kn} for basics on u.d.\ mod 1 and discrepancy). A foundational paper of Weyl \cite{weyl} connected the theory of uniform distribution with exponential sums, which allows a relatively easy proof of the fact that the sequence of fractional parts of $(p(n) \alpha)_{n \geq 1}$ is u.d.\ mod 1 if $\alpha \not\in \mathbb{Q}$; here $p$ is a non-constant polynomial with integer coefficients (to simplify the exposition, we will assume in the sequel that all sequences are understood as being reduced modulo 1; in other words, we will omit the phrase ``of fractional parts'' when speaking of real-valued sequences). In particular, the Kronecker sequence $(n \alpha)_{n \geq 1}$ is u.d.\ mod 1 if $\alpha \not\in \mathbb{Q}$. The pair correlation and the number variance of deterministic sequences can in principle also be studied using exponential sums, but this is usually a much more difficult problem than studying uniform distribution, due to the local-scale size of the test intervals, which requires the use of high-frequency terms in a Fourier analysis approach. Small-scale pseudorandomness is often seen as a ``finer'' property than large-scale pseudorandomness -- cf.\ for example the fact that Poissonian pair correlation implies uniform distribution mod 1, but not vice versa \cite{alp,gl,mark_man,stein}. In particular, the sequence of fractional parts $(n \alpha)_{n \geq 1}$ is u.d.\ mod 1 for $\alpha \not\in \mathbb{Q}$, but fails to have Poisson pair correlation \cite{ls} or Poissonian number variance \cite{ay} for any $\alpha$. The situation is expected to change for higher-degree polynomials. The Berry--Tabor conjecture implies that for a generic integrable system, such as a sequence $(n^d \alpha)_{n \geq 1}$ with $d \geq 2$ and ``unproblematic'' (in particular irrational) $\alpha$, the local fluctuations should follow a Poisson point process. A major breakthrough in this direction was achieved by Rudnick and Sarnak \cite{rs}, who proved that for any polynomial of degree $d \geq 2$, the pair correlation $(n^d \alpha)_{n \geq 1}$ is Poissonian for almost all $\alpha$. A fairly general metric result was later obtained by Aistleitner, Larcher and Lewko \cite{all}. However, a significant gap remains: while these and related results give a relatively complete picture of Poissonian behavior in the metrical sense (i.e., for a set of $\alpha$ of full Lebesgue measure), the problem remains wide open for any \textit{specific} individual value of $\alpha$ such as $\alpha = \sqrt{2}$, see \cite{hb,my,rsz}. The problem for the number variance is less studied and much less understood in the number-theoretic context than the problem for the pair correlation. Li and Yesha \cite{ly} obtained metric results which guarantee that the number variance of $(a_n \alpha)_{n \geq 1}$ is Poissonian for almost all $\alpha$, subject to a condition on the so-called additive energy of the initial segments of an integer sequence $(a_n)_{n \geq 1}$. Interestingly, stronger assumptions on the additive energy allow for a larger range of $s$ for which the asymptotics $V_N(s) \sim N s$ holds. The results of Li and Yesha were improved by Aistleitner and Yesha \cite{ay}, who also obtained metric results for general sequences $(a_n \alpha)_{n \geq 1}$, and in particular proved that the number variance of $(n^d \alpha)_{n \geq 1}$ is Poissonian for almost all $\alpha$, uniformly in the range $s \in [0, (\log N)^{-c}]$ for a suitable $c$.\\

The question concerning the maximal possible range of $s$, throughout which the asymptotics $V_N(s) \sim Ns$ is true for a generic realization of a sequence, arose from the work in \cite{ay}. There it was noted that this maximal possible range of $s$ seems to be unknown even in the i.i.d.\ random case, which should have served as the natural reference point for any deterministic sequence (Marklof's paper \cite{mark_nato} is a valuable source for the asymptotic behavior of the number variance, and the relation between number variance and pair correlation; however, the perspective taken there is rather the perspective of distributions, not of almost sure behavior). It was noted in \cite{ay} that the almost sure asymptotic $V_N(s) \sim N s$ cannot possibly be true throughout the full theoretically possible range $s \in [0,1/2]$ for a typical realization of the sequence, since for large $s$ (close to constant size) the random fluctuations become too large too frequently. More precisely, in \cite{ay} it was indicated (without formal proof) that for an i.i.d.\ random sequence of points it is true almost surely that 
$$
V_N(s) \sim N s \qquad \text{as $N \to \infty$, uniformly for $s \in [0, (\log N)^{-c}]$},
$$
where $c$ is a suitable constant, and a proof was given there that for such a sequence almost surely there are infinitely many $N$ such that
$$
V_N(s) \geq 1.001 N s \qquad \text{for $s = \frac{35}{\log \log N}$.}
$$
We will discuss a bit later why the asymptotics $V_N(s) \sim N s$ becomes more difficult when $s$ is larger. The main purpose of the paper is to determine the precise threshold where, for an i.i.d.\ random sequence, the asymptotics $V_N(s) \sim N s$ starts to fail for a generic realization. As a key ingredient for the proof of the almost sure behavior of $V_N(s)$ we also establish a very precise formula for the distribution of $V_N(s)$ for given $N$ and $s$.

\begin{thm} \label{th_1}
For $N \geq 1$, let $X_1, \dots, X_N$ be i.i.d.\ random variables having uniform distribution on $[0,1]$. Let $V_N(s)$ denote their number variance, as defined above. Then there exists a random variable $E_{N,s}$ such that the following holds. The distribution of $\frac{V_N(s)}{N}$ is the same as that of 
$$
\sum_{k=1}^\infty \lambda_k Z_k + E_{N,s},
$$
where $(Z_k)_{k \geq 1}$ are i.i.d.\ exponentially distributed random variables with parameter 1, and 
$$
\lambda_k = \frac{2 \sin^2(\pi ks)}{\pi^2 k^2 }, \quad k \geq 1.
$$
Furthermore, there is an absolute constant $c$ such that for $E_{N,s}$ we have
$$
\mathbb{P} \left( \left| E_{N,s} \right| \geq c \max \left\{ \frac{\sqrt{s} (\log N)^{3/2}}{\sqrt{N}}, \frac{(\log N)^2}{N} \right\} \right) = \mathcal{O} \left( N^{-100} \right) \qquad \text{as $N \to \infty$.}
$$
\end{thm}

Note that in the statement of the theorem, the sum $\sum_{k=1}^\infty \lambda_k Z_k = \sum_{k=1}^\infty \lambda_k(s) Z_k$ depends on the value of $s$, but it does not depend on $N$. The influence of $N$ on the distribution of $V_N(s)/N$ is entirely incorporated into the error term $E_{N,s}$. Vaguely speaking, the theorem thus asserts that the distribution of $V_N(s)/N$ is asymptotically essentially the same as that of $\sum_{k=1}^\infty \lambda_k(s) Z_k$ (at least as long as $s$ is of order at least $N^{-1+\varepsilon}$, say; for very small values of $s$ we will see that the error term can become dominant). 

\begin{thm} \label{th_2}
Let $(X_n)_{n \geq 1}$ be a sequence of i.i.d.\ random variables having uniform distri\-bution on $[0,1]$. Let $V_N(s)$ be the number variance of $X_1, \dots, X_N$. Let $g(N)$ be any positive function which decreases monotonically towards $0$. Then almost surely
$$
\frac{V_N(s)}{Ns} \to 1, \qquad \text{uniformly throughout the range } 0 \leq s \leq \frac{g(N)}{\log \log N}.
$$
On the other hand, for any fixed $\varepsilon > 0$, there exists an $\eta>0$ such that 
$$
 \left| \frac{V_N(s)}{Ns} - 1 \right| > \eta \qquad \text{for $s = \frac{\varepsilon}{\log \log N}$ and for infinitely many $N$}.
$$
\end{thm}

The reason why the asymptotics $V_N(s) \sim Ns$ becomes significantly more difficult to control as $s$ increases can be explained by looking at the variance of the number variance of a random sample. Above we already noted that for random points we have $\mathbb{E} (V_N(s)) = N s (1-s)$, which gives the asymptotic order $N s$ for small values of $s$; see equation \eqref{calculation_exp}. Our calculations will also show that the variance of $V_N(s)$ is of order $N^2 s^3$, see \eqref{variance_bound}. This means that $V_N(s)$ has a distribution which is concentrated around $Ns$ with fluctuations typically being of the size of the standard deviation, i.e.\ of order $N s^{3/2}$. This implies that $V_N(s)$ has better concentration around its mean when $s$ is smaller, making the problem of establishing $V_N(s) \sim Ns$ increasingly more difficult when $s$ becomes larger. The typical fluctuations of order $N s^{3/2}$  are almost surely infinitely often enhanced by further ``rare event'' fluctuations coming from the tail probabilities in the distribution of $V_N(s)$. When naively presuming a law of the iterated logarithm for these additional fluctuations, one should expect $V_N(s)$ to deviate from its average value by $\pm N s^{3/2} \sqrt{\log \log N}$ infinitely often, almost surely. To maintain the convergence $V_N(s) \sim Ns$, one would accordingly need $N s^{3/2} \sqrt{\log \log N} = o(Ns)$, or written in simplified form, $s = o(1/\log \log N)$. This gives a heuristic explanation why the critical threshold for the almost sure asymptotics should be expected at $s = o(1 / \log \log N)$. However, actually proving this heuristic is rather delicate, since the argument requires precise estimates for the tail probabilities for the distribution of $V_N(s)$; recall here that $V_N(s)$ is not just a sum of i.i.d.\ random variables, but that it is actually a functional (a square integral) of such a sum. An additional challenge lies in the dependence structure of $V_N(s)$ for different values of $s$ and of $N$.\\

\section*{Dedication}

This paper is dedicated to the memory of Vera T. S{\'o}s, whose foundational contributions to combinatorial and geometric discrepancy theory and the theory of uniform distribution modulo 1 continue to shape the field. A central theme in her work was the study of the fine-scale structure of sequences modulo one, and questions about the limitations of approximating continuous by discrete structures. Most famously, her proof of the ``Three Gaps Theorem'' \cite{sos} provided a complete description of the remarkable rigidity of the Kronecker sequence $(n \alpha)_{n \geq 1}$, showing that the gaps between successive elements can take at most three values -- this is closely connected to our discussion of the pair correlation / number variance of $(n \alpha)_{n \geq 1}$ above (she also wrote several other papers on Kronecker sequences, such as \cite{soskron}, which precisely quantifies the discrepancy of $(n \alpha)_{n \geq 1}$ in terms of the continued fraction expansion of $\alpha$). She had a deep interest in possible ways of qualifying and quantifying pseudo-random behavior, and the application of such pseudo-randomness tests to arithmetic or combinatorial structures; see in particular \cite{becksos,ls,ss}. Vera S{\'o}s was one of the most prominent representatives of the Hungarian school which linked combinatorics and probability theory in the second half of the twentieth century. Two of the authors of this paper (CA and SH) were raised by their PhD supervisor Istv\'an Berkes in the spirit of this magnificent Hungarian school. They wish to take this opportunity to thank Istv\'an for his dedication and support during their formative academic years.

\section{The KMT approximation} \label{sec_KMT}

We will use the famous Koml\'{o}s--Major--Tusn\'{a}dy theorem from empirical process theory. The result originally appeared in \cite{KMT1,KMT2}; for the wider framework we refer to \cite{CR} and \cite{CH}. In the statement of the result, and throughout the rest of the paper, we write $\mathbf{1}_A(x)$ for the indicator function of a set $A$.

\begin{lem}[Koml\'os--Major--Tusn\'ady approximation] \label{lem_kmt}
There exist absolute constants $c_1$, $c_2$, $c_3$, such that the following holds. For $N \geq 1$, let $X_1, \dots, X_N$ be i.i.d.\ random variables having uniform distribution on $[0,1]$. By possibly enlarging the underlying probability space,  there is a Brownian bridge $B_N(x),0 \leq x \leq 1$, such that for all $u > 0$
$$
\mathbb{P} \left( \sup_{0 \leq x \leq 1} \left| \left( \sum_{n=1}^N \mathbf{1}_{[0,x]}(X_n) - N x \right) - \sqrt{N} B_N(x) \right| > c_1 \log N + u \right) \leq c_2 e^{-c_3 u}.
$$
\end{lem}

We now use the KMT approximation to relate the distri\-bution of $V_N(s)$ to the distri\-bution of a certain (square-integral) functional of the Brownian bridge. Let $N \geq 1$ and $s>0$ be given. Let $X_1, \dots, X_N$ be i.i.d.\ random variables having uniform distribution on $[0,1]$. Then according to Lemma \ref{lem_kmt}, after possibly enlarging the underlying probability space, there exists a Brownian bridge process $B_N(x)$ and a stochastic process $D_N(x)$ such that
$$
D_N(x) = \left( \sum_{n=1}^N \mathbf{1}_{[0,x]}(X_n) - N x \right) - \sqrt{N} B_N(x),
$$
and such that for some absolute constant $c$ we have
\begin{equation} \label{g_size}
\mathbb{P} \left( \sup_{0 \leq x \leq 1} |D_N(x)| \geq c \log N \right) \leq N^{-100},
\end{equation}
say, for all sufficiently large $N$ (the exponent $100$ here is of course completely arbitrary; the only relevant point is that the exceptional probabilities arising from this step contribute to a convergent series when summing over $N$ in the application of the Borel--Cantelli lemma). Since in this paper we always interpret $[0,1]$ as the unit torus, it is natural for us to think of $B_N$ as a periodic function. This is unproblematic since for the Brownian bridge we have (almost surely) that $B_N(0)=B_N(1)=0$. Actually, we will be interested in the increments of this periodized Brownian bridge, so for $x \in [0,1]$ and $s \in [0,1]$, we define 
\begin{equation} \label{G_def}
G_N(x) = B_N(\{x + s\}) - B_N(x),
\end{equation}
where as before $\{ \cdot \}$ denotes the fractional part function. We note that $G_N$ of course depends on the value of the parameter $s$, but during this and the following section $s$ will be assumed to be fixed and this dependence will be suppressed. Clearly $G_N(x)$ is a Gaussian process, it admits a periodic extension to $\mathbb{R}$, and has mean zero.\\

Using the KMT approximation to express the number variance, we have
\begin{eqnarray}
&& V_N(s) \nonumber\\
& = & \int_0^1 \left( S_N(x,s) - Ns \right)^2 dx \nonumber\\
& = & \int_0^1 \left( \sum_{n=1}^N \mathbf{1}_{[0,\{x+s\}]}(X_n) - N (\{x+s\})- \sum_{n=1}^N \mathbf{1}_{[0,x]}(X_n) + Nx \right)^2 dx \nonumber\\
& = & \int_0^1 \left( \sqrt{N} B_N(\{x+s\}) +D_N(\{x+s\}) - \sqrt{N} B_N(x) - D_N(x)\right)^2 dx \nonumber\\
& = & \int_0^1 \left( \sqrt{N} G_N(x) +D_N(\{x+s\}) - D_N(x)\right)^2 dx \nonumber\\
& = & \int_0^1 \Big( N G_N(x)^2 + 2 \sqrt{N} G_N(x)  \left(D_N(\{x+s\}) - 
D_N(x) \right) + \left(D_N(\{x+s\}) - D_N(x) \right)^2 \Big) dx.\nonumber
\end{eqnarray}
Thus, according to what was said above, 
\begin{eqnarray}
&& \left| V_N(s) - N \int_0^1 G_N(x)^2 dx \right| \label{KMT_main}\\
& \leq & \int_0^1 \left| 2 \sqrt{N} G_N(x) \left(D_N(\{x+s\}) - D_N(x) \right) + \left(D_N(\{x+s\}) - D_N(x) \right)^2 \right| dx \nonumber\\
& \leq & 4  \|D_N\|_\infty \sqrt{N} \int_0^1 \left| G_N(x) \right| dx + 4 \|D_N\|_\infty^2 \label{KMT_error} \\
&\leq & 4  c \log N \sqrt{N} \int_0^1 \left| G_N(x) \right| dx + 4 c^2 (\log N)^2 \label{KMT_error_2}
\end{eqnarray}
with probability at least $1 - N^{-100}$, for sufficiently large $N$. Roughly speaking, what we have shown is that the value of $\frac{V_N(s)}{N}$ is essentially the same as that of 
\begin{equation} \label{brown_int}
H_N = \int_0^1 G_N(x)^2 dx = \int_0^1 \left( B_N(\{x + s\}) - B_N(x) \right)^2 dx,
\end{equation}
up to errors which are small with high probability. Among the error terms in lines \eqref{KMT_error} resp.\ \eqref{KMT_error_2}, the estimate for the size and exceptional probability of $\|D_N\|_\infty^2$ came from \eqref{g_size} and, consequently, directly from the KMT approximation in Lemma \ref{lem_kmt}. To estimate the size and exceptional probabilities for
$$
4 c \log N \sqrt{N} \int_0^1 \left| G_N(x) \right| dx,
$$
we will apply Cauchy--Schwarz to relate it to the distribution of the expression in line \eqref{brown_int}. So the main task now is to calculate the distribution of \eqref{brown_int}, since according to \eqref{KMT_main} this firstly constitutes the main term in the distribution of $V_N(s)$, and secondly also allows us to estimate the contribution of the remaining stochastic error term in \eqref{KMT_error_2}. This will be done in the next section, where we apply the Karhunen-Lo\`eve method to calculate the distribution of the mean-square increments of the Brownian Bridge.

\section{The Karhunen-Lo\`eve decomposition of the mean-square increments of the Brownian Bridge}

In the previous section we showed that for given $N$, the value of $\frac{V_N(s)}{N}$ can be approximated (up to errors which are very small with high probability) by $H_N$ from \eqref{brown_int}, which is in turn governed by the underlying Brownian Bridge process $B_N$ coming from the KMT approximation. Now it is crucial to note that while the path-wise approximation of $\frac{V_N(s)}{N}$ is tied to a specific Brownian Bridge $B_N$ via the equations surrounding \eqref{KMT_main}, and for different values of $N$ the number variance is connected with different Brownian Bridge processes $B_N$, for the \emph{distribution} of $\frac{V_N(s)}{N}$ it is not necessary to distinguish between different Brownian Bridges. To see this, note for example that the distribution of the expression in line \eqref{brown_int} is of course the same for all Brownian Bridges $B_N$. Accordingly, to understand the distribution of $\frac{V_N(s)}{N}$, it makes sense to define
\begin{equation} \label{G_def_without_N}
G(x) = B(\{x + s\}) - B(x) \qquad \text{and} \qquad H = \int_0^1 G(x)^2 dx,
\end{equation}
where $B$ is a (periodized) Brownian Bridge, and to study the distribution of $G$ and $H$ (where $G$ and $H$ of course depend on $s$, but this dependence is suppressed to simplify notations). This is the purpose of the present section.

\begin{lem}[Karhunen-Lo\`eve decomposition of $G$]  \label{lemma_decomp}
Let $G(x)$ be defined as in \eqref{G_def_without_N}. The distribution of $G(x)$ is the same as that of 
$$
\sum_{k=1}^\infty c_k\left( Y_k\sqrt{2}\sin(2k\pi x)+Y^\prime _k\sqrt{2}\cos(2k\pi x)\right),
$$
 where $Y_1,Y^\prime_1, Y_2, Y^\prime_2, \dots$ are independent standard Gaussian variables and $c_k=\frac{\sin(\pi ks)}{\pi k }$.
 %CA: where $Y_1, Y_2, \dots$ are i.i.d.\ random variables having exponential distribution with parameter $1$.  [Or should we include the coefficients as parameters of the exponential distributions?]
\end{lem}

For the proof of Lemma \ref{lemma_decomp} we use the Karhunen-Lo\`eve theorem in the following form (which is cited from \cite[Theorem 2]{Aletti2015}).
\begin{lem}[Karhunen-Loeve theorem for Gaussian process] \label{lemma_kl}

Let $S(t),~0 \leq t \leq 1 $ be a periodic zero-mean stochastic process with covariance $C(t,t^\prime)=\tilde{C}(t-t^\prime)$, where $\tilde{C}$ is a continuous real function; then in mean-square, uniformly in t,
$$S(t)=c_0Y^\prime_0+\sum_{k=1}^\infty c_k(Y_k \sqrt{2} \sin (2\pi kt)+Y^\prime_k\sqrt{2} \cos (2\pi kt)),$$
 where $(Y_k)_{k\geq 1}$ and $(Y_k^\prime)_{k\geq 0}$ are two independent sequences of independent standard Gaussian random variables, and where $(c_k)_{k\geq 0}\in \ell^2$ is such that
$$c_k^2=\int_0^1\tilde{C}(\theta)\cos(2 k \pi \theta) \,d\theta,\quad k=0,1,2,\dots $$
\end{lem}

\begin{proof}[Proof of Lemma \ref{lemma_decomp}]
We start by calculating 
$$
K_G(x,y)=\mathbb{E} \left[G(x)G(y)\right],
$$ 
the covariance kernel for the process $G$. A direct calculation, using the definition of $G$ and the fact that the covariance kernel of the Brownian Bridge is given by $K_B(s,t) = \min(s,t) - st$, yields
\begin{eqnarray*}
K_G(x,y) = s(1-s)-\min(s,\|x-y\|) =: \tilde{C}(x-y);
\end{eqnarray*} 
where $\|\cdot\|$ denotes the distance on the torus. Note in particular that this shows that 
$$
K_G(x,y) = K_G(\|x-y\|,0),
$$
i.e., the kernel $K_G$ is translation-invariant and its value only depends on the distance (on the torus) between $x$ and $y$. Note furthermore that the kernel has a symmetry in the sense that it gives the same value for $s$ and for $1-s$.\\

The next step is to solve the eigenvalue problem for this covariance kernel. Applying Lemma \ref{lemma_kl} to $G(x)$ we obtain $c_0^2=\int_0^1\tilde{C}(t) \,dt=0$ since $\mathbb{E} \left( \int_0^1 G(x) \, dx \right) = 0$, and for $k\geq 1$ we get
\begin{eqnarray*}
c_k^2& = & \int\limits_0^1 \tilde{C}(t)\cos(2\pi kt) \,dt \\
& = & 2\int\limits_0^s (s(1-s)-t)\cos(2\pi k t) \, dt-\int\limits_s^{1-s}s^2\cos(2\pi kt) \,dt\\
& = & \frac{1}{2\pi^2k^2}(1-2\pi s^2k\sin(2\pi ks)-\cos(2\pi k s)) -\frac{s^2}{\pi k}\cos(\pi k)\sin(\pi k(1-2s))\\
& = &  \frac{1-\cos(2\pi k s)}{2\pi^2k^2}=\frac{\sin^2(\pi ks)}{\pi^2 k^2 }.
\end{eqnarray*} 
\end{proof}
Our next concern will be the distribution of the $L^2$-norm of the process $G(x)$, i.e., the distribution of the random variable $H$ as defined in \eqref{G_def_without_N}.

As a corollary of Lemma \ref{lemma_decomp}, we have the following statement.
\begin{cor}  \label{cor_H}
    Let $Z_1, Z_2, \dots$ be i.i.d.\ random variables having an exponential distribution with parameter 1. Then the distribution of $H$ is the same as the distribution of 
    $$
    %\sum_{k=1}^\infty \nu_k(Y_k^2+(Y^\prime_k)^2),\quad \sum_{k=1}^\infty \nu_k X_k,\quad 
    \sum_{k=1}^\infty \lambda_k Z_k,
    $$
    where 
    $$
    \lambda_k = \frac{2 \sin^2(\pi ks)}{\pi^2 k^2 }, \quad k \geq 1.
    $$
\end{cor}

\begin{lem}  \label{lem_ex_var}
    We have $\mathbb{E}[H]=s(1-s)$ and $\operatorname{Var}[H]=\frac{2}{3}(2-3s)s^3$.
\end{lem}
\begin{proof}
    Using the fact that $H=\sum_{k=1}^\infty \lambda_k Z_k$ in distribution, and that $\mathbb{E}[Z_k]=1$ as well as $\operatorname{Var}[Z_k]=1$ for all $k$, we get
    \begin{align}
        \mathbb{E}[H]=\sum_{k=1}^\infty \lambda_k \quad \text{and}\quad \operatorname{Var}[H]=\sum_{k=1}^\infty \lambda_k^2. \label{sum_lambda_1}
    \end{align}
Recall that, according to the proof of Lemma \ref{lemma_decomp}, the eigenvalues of the kernel $K_G(x,y)$ are given by the pairs $\{\frac{\lambda_k}{2},\frac{\lambda_k}{2}\}_{k\geq 1}$. By the standard theory of integral operators, the sum of the eigenvalues equals the trace of the kernel, and the sum of the squares of the eigenvalues equals the $L^2$ norm of the kernel. Thus, 
\begin{equation*} \label{sum_lambda_2}
\mathbb{E}[H]=\int_0^1 K_G(x,x) ~dx=s(1-s)
\end{equation*}
and
\begin{equation}
\operatorname{Var}[H]=2\int_0^1\int_0^1 K_G(x,y)^2 ~dx ~dy=\frac{2}{3}s^3(2-3s),\quad \text{for}\quad 0\leq s\le 1/2.   \label{sum_lambda_3}
\end{equation}
The result for $s \in [1/2,1]$ follows from the symmetry of $K_G$ as a function of $s$, as noted in the proof of Lemma \ref{lemma_decomp}.
\end{proof}

\section{Estimates for the distribution of the mean-square increments of the Brownian bridge}

In the previous section, we gave an exact formula for the distribution of $H$ (the mean-square increments of the Brownian bridge) in terms of the distribution of a sum of independent random variables having exponential distribution. However, this representation does not allow us to directly estimate the actual values of the distribution function of $H$, and in particular it does not allow us to directly estimate the tail probabilities for large deviations of $H$ (which will be the key ingredient for the proof of Theorem \ref{th_2}). In the rest of this section, we analyze the distribution of $H$, to prepare for the proof of Theorem \ref{th_2} in Sections \ref{sec_th2} and \ref{sec_th3}.\\

To obtain upper and lower estimates for the tail probabilities of $H$ (and accordingly also of $V_N(s)$), we start with a computation of the moment-generating function of $H$.

\begin{lem}[MGF for H] \label{lemma_lower_H}
For $s \in (0,1/2)$ and $0<|t| \leq \frac{1}{3s^2}$ the moment generating function of $H-s(1-s)$ satisfies
$$
\mathbb{E} \left[e^{t (H-s(1-s))}\right] = e^{\frac{2 t^2 s^3}{3}-t^2s^4+\mathcal{O}(|t|^3s^5)},
$$
with an absolute implied constant (for which 16 is an admissible value). 
\end{lem}
\begin{proof}
    Recall that the moment-generating function of $Z\sim \operatorname{Exp}(1)$ is given by 
    $$\mathbb{E}[e^{uZ}]=\frac{1}{1-u}$$
    for $u<1$. Recall also that $\lambda_k \leq 2s^2$ for all $k \geq 1$, and that $\sum_{k=1}^\infty \lambda_k = s(1-s)$. Accordingly, for $0<|t|<\frac{1}{3s^2}$ we have $|t \lambda_k|<1$ for all $k\geq 1$, which implies 
    \begin{eqnarray*}
        \mathbb{E} [e^{t (H-s(1-s))}]&=&\mathbb{E} [e^{t\sum_{k=1}^\infty \lambda_k Z_k}]\cdot e^{-t\sum_{k=1}^\infty \lambda_k}\\
        & = &\left(\prod_{k=1}^\infty \frac{1}{1-t\lambda_k}\right)\cdot e^{-t\sum_{k=1}^\infty \lambda_k} \\
        & = &e^{-\sum_{k=1}^\infty \log(1-t\lambda_k)}\cdot e^{-t\sum_{k=1}^\infty \lambda_k} .
    \end{eqnarray*}
    Using Taylor expansion for $\log(1-t\lambda_k)$ we get 
    $$ \mathbb{E} [e^{t (H-s(1-s))}]=e^{\sum_{n=2}^\infty \frac{t^n}{n}\sum_{k=1}^\infty \lambda_k^n}.$$
    Recall that $\lambda_k \leq \min \{2s^2, 2 / (\pi k)^2\}$ for all $k \geq 1$. Hence, splitting the range for $k$ into $k \leq \lceil 1/(\pi s) \rceil$ and $k > \lceil 1/(\pi s) \rceil$, we have
    $$0\leq\sum_{k=1}^\infty \lambda_k^n\leq \left\lceil \frac{1}{\pi s} \right\rceil (2s^2)^n+\sum_{k > \lceil 1/(\pi s) \rceil} \left(\frac{2}{\pi^2k^2}\right)^n\leq 2^n s^{2n-1}+\frac{2^n}{\pi (2n-1)}s^{2n-1}\leq 2^{n+1} s^{2n-1}, $$
    where we used that $\lceil x \rceil \leq \frac{\pi}{2} x$ for $x \geq 2/\pi$ (recall that $s \leq 1/2$ by assumption). This gives us 
    $$\left|\sum_{n=3}^\infty\frac{t^n}{n}\sum_{k=1}^\infty \lambda_k^n \right|\leq \frac{2}{s}\sum_{n=3}^\infty \frac{(2|t|s^2)^n}{n} \leq \frac{2}{3s}\sum_{n=3}^\infty (2|t|s^2)^n= \frac{16 |t|^3 s^5}{3} \frac{1}{1-2 |t| s^2} \leq 16 |t|^3 s^5,$$
    since $|t| s^2 \leq 1/3$ by assumption. For $n=2$, we already know that 
    $$\sum_{k=1}^\infty\lambda_k^2=\frac{4}3s^3-2s^4.$$
    Overall, this yields
    $$
\mathbb{E} [e^{t (H-s(1-s))}] = e^{\frac{2 t^2 s^3}{3}-t^2s^4+\mathcal{O}(|t|^3s^5)},
$$
where the constant hidden in the $\mathcal{O}$-term is at most 16.
\end{proof}

\begin{lem}[Tail probability upper bound from MGF] \label{lem_H_tail}
For $s \in (0, 1/2]$ and $x \in [0, s/100]$, we have
\begin{equation} \label{H_upper_bound}
\mathbb{P} \left( |H - s(1-s)| > x \right) \leq 2 e^{-\frac{x^2}{4 s^3}}.
\end{equation}
\end{lem}

\begin{proof}
The Chernoff bound asserts that for any $t > 0$ we have
\begin{equation} \label{chernoff}
\mathbb{P}(H - s(1-s) > x) \leq e^{-tx} \mathbb{E}[e^{t(H-s(1-s))}].
\end{equation}
Setting $t = \frac{x}{2s^3}$, by Lemma \ref{lemma_lower_H} (and using the fact that the implied absolute constant can be chosen to be 16) we have
$$
 \mathbb{E}[e^{t(H-s(1-s))}] \leq \exp \left(\frac{2 t^2 s^3}{3} + 16 t^3 s^5 \right)  \leq \exp \left(\frac{x^2}{6s^3} + \frac{16x^2}{800 s^3} \right) \leq \exp \left(\frac{x^2}{4s^3}\right),
$$
where we used the assumption that $x \leq s/100$. Thus we have
\begin{equation} \label{one_sided}
\mathbb{P}(H - s(1-s) > x) \leq \exp \left( -\frac{x^2}{2 s^3} + \frac{x^2}{4s^3} \right) \leq \exp \left( -\frac{x^2}{4 s^3}\right).
\end{equation}
Switching from $H - s(1-s)$ to $|H - s(1-s)|$ (at the cost of a factor 2) is a standard argument relying on symmetry: changing from $H - s(1-s)$ to $-H + s(1-s)$ does not affect the proof of Lemma \ref{lemma_lower_H}, and one obtains 
$$
\mathbb{E} \left[e^{t (-H+s(1-s))}\right]= e^{\frac{2 t^2 s^3}{3}-t^2s^4+\mathcal{O}(t^3s^5)},
$$
in perfect analogy to the conclusion of Lemma \ref{lemma_lower_H}. Using the Chernoff bound now yields
$$
\mathbb{P}(-H + s(1-s) > x) \leq \exp \left( -\frac{x^2}{4 s^3}\right)
$$
as an analogue of \eqref{one_sided}, so that overall
$$
\mathbb{P}(|H - s(1-s)| > x) \leq 2 \exp \left( -\frac{x^2}{4 s^3}\right),
$$
as claimed.
\end{proof}

\begin{lem}[Tail probability lower bound from MGF] \label{lemma_lower_tail}
For any $s \leq 1/100$ and $\varepsilon \leq 1/1000$ we have
$$
\mathbb{P} \left( H - s(1-s) \geq \varepsilon s \right) \geq  \min \left\{ \frac{\varepsilon^4}{s^2}, \frac{1}{4} \right\} e^{-\frac{17 \varepsilon^2}{s}}.
$$
\end{lem}

\begin{proof}
Recall the Paley--Zygmund inequality, which asserts that for any non-negative random variable $Y$ with finite variance, and for any $0 \leq \theta \leq 1$, we have
$$
\mathbb{P} \left( Y \geq \theta \mathbb{E} [Y] \right) \geq (1-\theta)^2 \frac{\mathbb{E}[Y]^2}{\mathbb{E}[Y^2]}.
$$
Applying this to $Y = e^{t (H-s(1-s))}$ for some $t > 0$, we obtain
$$
\mathbb{P} \left( e^{t(H-s(1-s))} \geq \theta \mathbb{E} [e^{t (H-s(1-s))}] \right) \geq (1-\theta)^2 \frac{\mathbb{E} [e^{t (H-s(1-s))}]^2}{\mathbb{E} [e^{2t (H-s(1-s))}]}.
$$
By Lemma \ref{lemma_lower_H}, using the fact that the implied absolute constant can be chosen to be 16, we have
$$
\mathbb{P} \left( e^{t(H-s(1-s))} \geq \theta \mathbb{E} \left[e^{t (H-s(1-s))} \right] \right) \geq (1-\theta)^2 \frac{e^{\frac{4 t^2 s^3}{3} - 2t^2 s^4 - 32 t^3 s^5}}{e^{\frac{8 t^2 s^3}{3} - 4 t^2 s^4 +  128 t^3 s^5}} \geq (1-\theta)^2 e^{- \frac{4}{3} t^2 s^3 - 160 t^3 s^5},
$$
which is the same as 
$$
\mathbb{P} \left( t (H - s(1-s)) \geq \log \theta + \log \mathbb{E} \left[e^{t (H-s(1-s))}\right] \right) \geq (1-\theta)^2 e^{-\frac{4}{3} t^2 s^3 - 160 t^3 s^5};
$$
more precisely, this is true under the supposition that we will choose a value of $t$ which is within the admissible range for the application of Lemma \ref{lemma_lower_H} (which indeed we will). From Lemma \ref{lemma_lower_H} we also know that $\mathbb{E} [e^{t (H-s(1-s))}] \geq e^{\frac{2 t^2 s^3}{3} - t^2 s^4 - 16 t^3 s^5}$. Accordingly, choosing $t = \frac{3 \varepsilon}{s^{2}}$,  we have
\begin{eqnarray*}
\log \mathbb{E} \left[e^{t (H-s(1-s))}\right] & \geq & \frac{2 t^2 s^3}{3} - t^2 s^4 - 16 t^3 s^5 \\
& \geq & \frac{2 t^2 s^3}{3} - \frac{t^2 s^3}{100} - 48 \varepsilon t^2 s^3 \\
& \geq & \frac{2 t^2 s^3}{3} - \frac{t^2 s^3}{100} - \frac{48}{1000} t^2 s^3 \geq \frac{t^2 s^3}{2} = \frac{3 \varepsilon s t}{2},
\end{eqnarray*}
since $s \leq 1/100$ and $\varepsilon \leq 1/1000$ by assumption; note that with this choice of $t$ we have $t<1/(3s^2)$, as required for the application of Lemma \ref{lemma_lower_H}. Now we pick 
$$
\theta = e^{-\frac{t^2 s^3}{6}} = e^{-\frac{t \varepsilon s}{2}} = e^{-\frac{3\varepsilon^2}{2s}},
$$
and using $(1 - e^{-x})^2 \geq \min \{x^2/2, 1/4\}$ for all $x \geq 0$ (which can be easily checked using basic calculus) we get (after canceling a factor $t$ on both sides of the expression inside of $\mathbb{P} (\dots)$)
$$
\mathbb{P} \left( H - s(1-s) \geq -\frac{\varepsilon s}{2} + \frac{3 \varepsilon s}{2} \right) \geq  \min \left\{ \frac{9 \varepsilon^4}{8s^2}, \frac{1}{4} \right\} e^{- \frac{4}{3} t^2 s^3 - 160 t^3 s^5}.
$$
Inserting $t = \frac{3 \varepsilon}{s^{2}}$ gives $\frac{4}{3} t^2 s^3= \frac{12 \varepsilon^2}{s}$ and using that $160 t^3 s^5 = 480 \varepsilon t^2 s^3 \leq \frac{t^2 s^3}{2} = \frac{9 \varepsilon^2}{2s}$ by our assumption on the size of $\varepsilon$, we finally arrive at
$$
\mathbb{P} \left( H - s(1-s) \geq \varepsilon s \right) \geq  \min \left\{ \frac{9 \varepsilon^4}{8s^2}, \frac{1}{4} \right\} e^{-\frac{12 \varepsilon^2}{s}-\frac{9 \varepsilon^2}{2s}} \geq \min \left\{ \frac{\varepsilon^4}{s^2}, \frac{1}{4} \right\} e^{-\frac{17 \varepsilon^2}{s}} ,
$$
as desired.
\end{proof}

\section{Proof of Theorem \ref{th_1}}

According to the calculations culminating in \eqref{KMT_error_2}, for any $N$ and any $s$ there is a Brownian Bridge $B_N(x)$, together with a process $G_N(x)$ and a random variable $H_N$ as defined in \eqref{G_def} and \eqref{brown_int}, such that we can write
$$
\frac{V_N(s)}{N} = H_N + E_{N,s},
$$
where according to \eqref{KMT_error_2}
$$
|E_{N,s}| \leq 4  \frac{c \log N}{\sqrt{N}} \int_0^1 \left| G_N(x) \right| dx + \frac{4c^2 (\log N)^2}{N}
$$
with probability at least $1-N^{-100}$. By Cauchy--Schwarz we have
$$
\int_0^1 \left| G_N(x) \right| dx \leq \sqrt{\int_0^1 G_N(x)^2 dx} = \sqrt{H_N}. 
$$
Of course the distribution of $H_N$ is the same for all $N$, so we can simply write $H$ as defined in \eqref{G_def_without_N}. We pick $t = \frac{1}{3s^2}$ (the maximal admissible value in Lemma \ref{lemma_lower_H}) and $x = 303 s  \log N$, and using the Chernoff bound \eqref{chernoff} we obtain
$$
\mathbb{P} \left( H - s(1-s) \geq 303 s \log N \right) \leq \exp \left(-\frac{x}{3s^2} \right) \exp\left(\frac{2}{27 s} + \frac{16}{27 s} \right) \leq N^{-100}
$$
for sufficiently large $N$ (note that Lemma \ref{lem_H_tail} is not applicable with the values of $x$ and $s$ needed here, which is why we revert to Lemma \ref{lemma_lower_H} and the Chernoff bound). Accordingly,
$$
\mathbb{P} \left( \sqrt{H} \geq \left(\sqrt{303} + 1 \right) \sqrt{s \log N} \right) = \mathcal{O} \left(N^{-100} \right),
$$
and so 
$$
\mathbb{P} \left( |E_{N,s}| > 74 \frac{c \sqrt{s} (\log N)^{3/2}}{\sqrt{N}} + \frac{4 c^2(\log N)^2}{N} \right) = \mathcal{O} \left( N^{-100} \right). 
$$
This proves Theorem \ref{th_1}. 

\section{Almost sure behavior of $V_N(s)$. Proof of Theorem \ref{th_2}, convergence part.} \label{sec_th2}

In this section we prove the convergence part of Theorem \ref{th_2}. More precisely, as throughout the paper we write $V_N(s)$ for the number variance of i.i.d.\ uniform random points in $[0,1]$. We assume that $g(N)$ is any function which is monotonically decreasing towards $0$. Then the aim is to prove that 
$$
\frac{V_N(s)}{Ns} \to 1, \qquad \text{uniformly throughout the range } 0 \leq s \leq \frac{g(N)}{\log \log N},
$$
almost surely. 

\begin{lem} \label{lemma_lip1}
For any $N \geq 1$ and any $s,s' \in [0,1]$ with $s < s'$ we have 
$$
V_N (s) \leq V_N (s') + N^2 \left( (s')^2 - s^2 \right).
$$
\end{lem}

We note that Lemma \ref{lemma_lip1} is not a probabilistic statement; it holds for any (deterministic) configuration $X_1, \dots, X_N$ of points in $[0,1]$. 

\begin{proof}
Recall the notation from \eqref{S_def} and the definition of $V_N$ in \eqref{V_def}. Since $s < s'$, we trivially have $S_N(x,s) \leq S_N(x,s')$. Accordingly, 
\begin{eqnarray*}
    V_N(s) & \leq & \int_0^1 S_N(x,s')^2 - N^2 s^2 ~dx \\
    & = & \int_0^1 S_N(x,s')^2 - N^2 (s')^2 ~dx + N^2 ((s')^2-s^2) \\
    & = & V_N(s') + N^2 ((s')^2-s^2).
\end{eqnarray*}
\end{proof}

\begin{lem} \label{lemma_lip2}
For any $N, N' \geq 1$ with $N < N'$ and any $s \in [0,1]$ we have
$$
V_N (s) \leq V_{N'} (s) + \left((N')^2 - N^2\right) s^2.
$$
\end{lem}

Again, Lemma \ref{lemma_lip2} is not a probabilistic statement, it holds for any (deterministic) configuration of points. 

\begin{proof}
Since $N < N'$, we trivially have $S_N(x,s) \leq S_{N'}(x,s)$. Accordingly, 
\begin{eqnarray*}
    V_N(s) & \leq & \int_0^1 S_{N'}(x,s)^2 - N^2 s^2 ~dx \\
    & = & \int_0^1 S_{N'}(x,s)^2 - (N')^2 s^2 ~dx + \left((N')^2- N^2\right) s^2. \\
    & = & V_{N'}(s) + \left((N')^2- N^2\right) s^2.
\end{eqnarray*}
\end{proof}

\begin{lem}[``Microscopic regime''] \label{lem_th2_small}
We have 
$$
\frac{V_N(s)}{Ns} \to 1, \qquad \text{uniformly throughout the range } 0 \leq s \leq N^{-3/4},
$$
almost surely. 
\end{lem}

Lemma \ref{lem_th2_small} is the first part of the proof of the convergence part of Theorem \ref{th_2}, for the range where $s$ is ``very small''. It is necessary to split off this part of the range of $s$, since the error terms from the KMT approximation (whose size is essentially described by the random variable $E_{N,s}$ in the statement of Theorem \ref{th_1}) become dominant in comparison with the desired main term (which is $\frac{V_N(s)}{N} \sim s$) when $s$ becomes smaller in order than $(\log N)^2 / N$. However, as we have seen the problem of establishing $V_N(s) \sim Ns$ becomes gradually easier when $s$ is smaller, so to prove Lemma \ref{lem_th2_small} for the range where $s$ is very small we do not need the full power of Theorem \ref{th_1} and the exponential tail bounds from Lemma \ref{lem_H_tail}. Instead, it suffices to work the first and second moments of $V_N(s)$, which can be calculated relatively easily by combinatorial methods.\\

\begin{proof}[Proof of Lemma \ref{lem_th2_small}] Let $X_1, X_2, \dots$ be i.i.d.\ random variables having uniform distribution on $[0,1]$. Then by definition
$$
V_N(s) = \int_0^1 \left(\sum_{n=1}^N \mathbf{1}_{[x,x+s]}(X_n) \right)^2~dx - N^2 s^2.
$$
The expected value of $V_N(s)$ is given by
\begin{align} \label{calculation_exp}
    \mathbb{E}[V_N(s)] =\int\limits_0^1 \mathbb{{E}}\left(\sum_{n=1}^N \mathbf{1}_{[x,x+s]}(X_n) \right)^2-N^2s^2~dx=Ns(1-s),
\end{align}
where we used that the integrand can be viewed as the variance of a binomial random variable $\operatorname{Bin}(N,s)$.
%Expanding the square and interchanging integration with summation, this gives
%$$ V_N(s) = \sum_{m,n=1}^N \int_{0}^1 \mathbf{1}_{[x,x+s]}(X_m) \mathbf{1}_{[x,x+s]}(X_n)~dx - N^2 s^2. $$
%Now
%$$ \int_{0}^1 \mathbf{1}_{[x,x+s]}(X_m) \mathbf{1}_{[x,x+s]}(X_n)~dx = \max \left\{s- \|X_m - X_n\|, 0  \right\}. $$
%When $m=n$, then clearly the expression in the line above is $s$. If $m \neq n$, then (by independence) $X_m - X_n$ has the same distribution as $X_1$, and accordingly
%\begin{eqnarray} 
%\mathbb{E} [V_N(s)] & = & N s + N (N-1) \mathbb{E} \left( \max \left\{s- \|X_1\|, 0  \right\} \right) - N^2 s^2 \nonumber\\
%& = & Ns + N(N-1) s^2 - N^2 s^2 \nonumber \\
%& = & N s (1-s). \label{calculation_exp}
%\end{eqnarray}
Using properties of the variance, we can further calculate
\begin{align*}
\mathrm{Var}[V_N(s)]&=\int_0^1\int_0^1\mathrm{Cov}\left(\left(\sum_{n=1}^N \mathbf{1}_{[x,x+s]}(X_n)\right)^2,\left(\sum_{n=1}^N \mathbf{1}_{[y,y+s]}(X_n)\right)^2\right)~dx~dy\\
&=\sum_{n_1,m_1=1}^N\sum_{n_2,m_2=1}^N \int_0^1\int_0^1I(m_1,n_1,m_2,n_2,x,y)~dx~dy,
\end{align*}
where $I\eqqcolon I(m_1,n_1,m_2,n_2,x,y)$ denotes
\begin{align*}
\mathbb{P}(X_{m_1},X_{n_1}\in [x,x+s],&X_{m_2},X_{n_2}\in [y,y+s]) \\ &\quad -
\mathbb{P}(X_{m_1},X_{n_1}\in [x,x+s])\mathbb{P}(X_{m_2},X_{n_2}\in [y,y+s]).
\end{align*}
%\begin{eqnarray*}
%\operatorname{Var}[V_N(s)] & = & \mathbb{E} \left[ \left( \sum_{\substack{1 \leq m,n \leq N}} \int_{0}^1 \mathbf{1}_{[x,x+s]}(X_m) \mathbf{1}_{[x,x+s]}(X_n)~dx -N^2s^2\right)^2 \right] -\mathbb{E}[V_N(s)]^2\\
%& = & \mathbb{E} \left[  \sum_{\substack{1 \leq m_1,m_2,n_1,n_2 \leq N}} \int_{0}^1 \int_{0}^1\mathbf{1}_{[x,x+s]}(X_{m_1}) \mathbf{1}_{[x,x+s]}(X_{n_1}) \right. \\
%& & \qquad \qquad \qquad \qquad \left. \phantom{\sum_{\substack{1 \leq m_1,m_2,n_1,n_2 \leq N}}}  \mathbf{1}_{[y,y+s]}(X_{m_2}) \mathbf{1}_{[y,y+s]}(X_{n_2})~dx ~dy \right] \\& & \qquad - 2N^2s^2\mathbb{E}\left[\sum_{\substack{1 \leq m,n \leq N}} \int_{0}^1 \mathbf{1}_{[x,x+s]}(X_m) \mathbf{1}_{[x,x+s]}(X_n)~dx \right] \\
%& & \qquad +N^4s^4-N^2s^2(1-s)^2
%\\ & = & \sum_{\substack{1 \leq m_1,m_2,n_1,n_2 \leq N}} \int_{0}^1 \int_{0}^1 I(x,y,m_1,n_1,m_2,n_2) ~dx ~dy \\ & & \qquad-2N^2s^2\left(Ns(1-s)+N^2s^2\right)+N^4s^4-N^2s^2(1-s)^2,
%\end{eqnarray*}
%where $I$ denotes
%$$\mathbb{E}\left[\mathbf{1}_{[x,x+s]}(X_{m_1}) \mathbf{1}_{[x,x+s]}(X_{n_1})\mathbf{1}_{[y,y+s]}(X_{m_2}) \mathbf{1}_{[y,y+s]}(X_{n_2})\right].$$
Note that whenever $\{m_1,n_1\}\cap\{m_2,n_2\}=\emptyset$, we have $I=0$, because of independence of the pairs $(m_1,n_1)$ and $(m_2,n_2)$. 

%$$\mathbb{E} \left[\mathbf{1}_{[x,x+s]}(X_k)\mathbf{1}_{[y,y+s]}(X_k) \right]=\lambda([x,x+s]\cap[y,y+s])=\max\{s-\|x-y\|,0\}.$$

We need to distinguish several different cases, depending on the properties of the 4-tuple $(m_1,n_1,m_2,n_2)$. Before doing that, note that the length of $[x,x+s]\cap [y,y+s]$ is given by $$\max\{s-\|x-y\|,0\}.$$

Moreover, we have $$ \int_0^1\int_0^1\max\{s-\|x-y\|,0\}\,~dx ~dy= s^2$$  
and 
$$\int_0^1\int_0^1\left(\max\{s-\|x-y\|,0\}\right)^2\,~dx ~dy= \frac{2}{3}s^3.$$
Cases when $|\{m_1,n_1\}\cap\{m_2,n_2\}|=1$:\\

\begin{itemize}
\item If $m_1=n_1=m_2=n_2$, then $$I=\max\{s-\|x-y\|,0\}-s^2,$$ and hence $\int_0^1\int_0^1 I~dx~dy=0$.\\
\item If $m_i=n_i=m_j\neq n_j$, then $$I=\max\{s-\|x-y\|,0\}s-s^3,$$ and hence $\int_0^1\int_0^1 I~dx~dy=0$.\\

\item If $m_1 \neq n_1$ and $m_2 \neq n_2$ with $\{m_1,n_1\} \cap \{m_2, n_2\} = 1$, then 
$$ I=\max\{s-\|x-y\|,0\}s^2-s^4, $$
and hence  $\int_0^1\int_0^1 I~dx~dy=0$.\\

\end{itemize}
Finally, if $|\{m_1,n_1\}\cap\{m_2,n_2\}|=2$, then 
$$
I=\left(\max\{s-\|x-y\|,0\}\right)^2-s^4.
$$
The number of such 4-tuples is $2 N (N-1)$. \\

Overall, this gives
\begin{eqnarray}
    \operatorname{Var}[V_N(s)] &=& 2N(N-1) \left(\frac{2}{3}s^3-s^4\right) \nonumber \\
    & = & N(N-1)\cdot \frac{2 s^3(2-3s)}{3} \label{variance_size}\\
    & \leq & \frac{4}{3} N^2 s^3. \label{variance_bound}
\end{eqnarray}
Note how \eqref{variance_size} matches together (after re-scaling) with \eqref{sum_lambda_3}. In probabilistic language, the term \eqref{variance_size} is the variance of a degenerate U-statistic of order 2. Generally, such U-statistic variances are typically intricate to calculate, but in our particular situation, the kernel function -- defined by the measure of the intersection of two intervals -- has a very simple geometric structure. This simplicity allows for a direct combinatorial calculation with reasonable effort, bypassing the need for complex spectral decompositions at this stage. Note also how the leading-order terms $N^3$ and $N^4$ cancel out to yield \eqref{variance_size}, which is typical for the variance of degenerate U-statistics. This pertains to the Hoeffding decomposition of $V_N$, and reflects the fact that the number variance is ``purely quadratic'' (in the language of U-statistics) -- it depends only on the pairwise interactions of points, not on their individual positions. We note in passing that the use of U-statistics to interpret number-theoretic aspects of local statistics, such as the pair correlation, appeared already in the work of Berkes, Philipp and Tichy \cite{bpt}.\\

Let $k \geq 1$ be given. We will consider all integers of the form
\begin{equation} \label{jvw1}
N_j = 2^{k-1} + j \lceil 2^{2k/3} \rceil, \qquad 0 \leq j \leq \lceil 2^{k/3-1} \rceil.
\end{equation}
Furthermore, we will consider all values of $s$ of the form
\begin{equation} \label{jvw2}
s_v = \frac{1}{2^{3k/4}} \left( 1 + \frac{1}{2^{k/3}} \right)^{-v}, \qquad  v = 0,1,2, \dots
\end{equation}
Note that here we study infinitely many values of $s_v$, starting from $s_0 = 2^{-3k/4}$ (which is of order roughly $N_j^{-3/4}$) all the way down to arbitrarily small (``microscopic'') values of $s_v$. It will be possible to take a union bound over an infinite index set here since (as discussed above) the problem becomes easier when $s$ becomes smaller, and it will turn out that the arising exceptional probabilities associated with the discretizations points $s_v$ can be summed over $v=0,1,2,\dots$.\\

We define sets
$$
A_{j,v} = \left\{ \left| V_{N_j}(s_v) - N_j s_v (1-s_v) \right| > k^{-1} N_j s_v\right\}.
$$
Then, according to \eqref{variance_bound} and Chebyshev's inequality, we have
$$
\mathbb{P} (A_{j,v}) \leq \frac{4 k^2 s_v}{3} = \mathcal{O} \left(\frac{k^2}{2^{3k/4}} \left( 1 + \frac{1}{2^{k/3}} \right)^{-v} \right).
$$
By the union bound and the formula for the geometric series this yields
\begin{equation} \label{borel_1}
\mathbb{P} \left(\bigcup_{0 \leq j \leq \lceil 2^{k/3-1} \rceil} \bigcup_{v \geq 0} A_{j,v} \right) = \mathcal{O} \left(k^2 2^{-k/12} \right).
\end{equation}
Clearly the term $k^2 2^{-k/12}$ on the right-hand side of \eqref{borel_1} is summable in $k$. Accordingly, by the Borel--Cantelli lemma, with probability one there are only finitely many values of $k$ such that an event $A_{j,v}$ occurs (for any choice of $j,v$ which is admissible according to the construction in equations \eqref{jvw1} and \eqref{jvw2}).\\

This discretization step, together with the application of Chebyshev's inequality and the Borel--Cantelli lemma, was the crucial part in the proof of Lemma \ref{lem_th2_small}. It remains to carry out an approximation argument, which allows us to pass from  $V_N(s)$ for special values of $N$ and $s$ (namely $N$ of the form $N_j$ and $s$ of the form $s_v$) to the general case. Assume that we have some $N \in \{2^{k-1}, \dots, 2^k\}$ for a value of $k$ such that none of the events $A_{j,v}$ (for any values of $j,v$ that are admissible for such $k$) has occurred. Assume that $s \leq 2^{-3k/4}$. Then there exists a value of $j$ such that $N_j \leq N \leq N_{j+1}$, together with a value of $v$ such that $s_{v+1} \leq s \leq s_{v}$ (note that $s_v$ is decreasing as a function of $v$). By construction $N_{j+1} - N_j \leq \lceil 2^{2k/3} \rceil$ and $s_{v} - s_{v+1} = 2^{-k/3} s_{v+1} \leq 2^{-k/3} s$. Thus, according to Lemmas \ref{lemma_lip1} and \ref{lemma_lip2}, we have
\begin{eqnarray}
V_N(s) & \leq & V_{N_{j+1}} (s) + (N_{j+1}^2 - N^2) s^2 \label{chain_1}\\
& \leq & V_{N_{j+1}} (s_{v}) + N_{j+1}^2 (s_{v}^2 - s^2) + (N_{j+1}^2 - N^2) s^2 \nonumber\\
& \leq & V_{N_{j+1}} (s_{v}) + 2N_{j+1}^2 (s_{v} - s) s_{v} + 2(N_{j+1} - N) N_{j+1} s^2 \nonumber\\
& \leq & V_{N_{j+1}} (s_{v}) + \mathcal{O} \left(2^{k} N 2^{-k/3} s 2^{-3k/4} + 2^{2k/3} N 2^{-3k/4} s \right) \nonumber\\
& = & V_{N_{j+1}} (s_{v}) + \mathcal{O} \left(2^{-k/12} N s\right). \nonumber
\end{eqnarray}
Since we assumed that $A_{j+1,v}$ has not occurred, we have
$$
\left| V_{N_{j+1}} (s_{v}) - N_{j+1} s_{v} \right| \leq k^{-1} N_{j+1}  s_{v} + N_{j+1} s_v^2. 
$$
Since 
\begin{eqnarray*}
N_{j+1} s_{v} - N s & = & N_{j+1} (s_{v} - s) + s (N_{j+1} - N) \\
& = & \mathcal{O} \left( N 2^{-k/3} s + s N 2^{-k/3} \right) \\
& = & \mathcal{O} \left(2^{-k/3} Ns \right), 
\end{eqnarray*}
and since $N_{j+1} s_v^2 = \mathcal{O} (Ns 2^{-3k/4})$, overall this yields
\begin{equation} \label{vn_upper}
V_N(s) \leq N s + \mathcal{O} \left( k^{-1} N s \right),
\end{equation}
where the implied constant is independent of $N$, $s$ and $k$. In the same way we can give a corresponding lower bound for $V_N(s)$, by comparing it to $V_{N_j} (s_{v+1})$. To get this, we start with 
$$
V_N(s)  \geq V_{N_{j}} (s) - (N_{j}^2 - N^2) s^2
$$
(from Lemma \ref{lemma_lip2}), and then follow a chain of inequalities similar to those in \eqref{chain_1}--\eqref{vn_upper} (using also Lemma \ref{lemma_lip1} at the suitable location to switch from $V_{N_j}(s)$ to $V_{N_j}(s_{v+1})$) to obtain 
\begin{equation} \label{vn_lower}
V_N(s) \geq N s - \mathcal{O} \left( k^{-1} N s \right),
\end{equation}
with an implied constant that is independent of $N$, $s$ and $k$. As deduced above from the application of the Borel--Cantelli lemma, almost surely both inequalities \eqref{vn_upper} and \eqref{vn_lower} hold (uniformly in the relevant range of $s$) for all sufficiently large $k$, i.e.\ for all sufficiently large $N$. Accordingly, overall we have shown that almost surely
$$
\frac{V_N(s)}{Ns} \to 1, \qquad \text{uniformly for $s \in (0 , N^{-3/4}]$,}
$$
as claimed.
\end{proof}

Now we come to the part of Theorem \ref{th_2} which covers the larger values of $s$. Mathematically this is the more difficult range, which requires the exponential tail estimates which we derived in Lemma \ref{lem_H_tail} as a consequence of Theorem \ref{th_1}. We split the full range $[N^{-3/4}, \frac{g(N)}{\log \log N}]$ for $s$ into the union of $[N^{-3/4}, (\log N)^{-2}]$ and $[(\log N)^{-2}, \frac{g(N)}{\log \log N}]$, which the last one being the most delicate one.

\begin{lem}[``Mesoscopic Regime''] \label{lem_th2_large_a}
We have 
$$
\frac{V_N(s)}{Ns} \to 1, \qquad \text{uniformly throughout the range } N^{-3/4} \leq s \leq \frac{1}{(\log N)^2},
$$
almost surely. 
\end{lem}

\begin{lem}[``Macroscopic Regime''] \label{lem_th2_large_b}
We have 
$$
\frac{V_N(s)}{Ns} \to 1, \qquad \text{uniformly throughout the range } \frac{1}{(\log N)^2} \leq s \leq \frac{g(N)}{\log \log N},
$$
almost surely. 
\end{lem}

For the proof of Lemmas \ref{lem_th2_large_a} and \ref{lem_th2_large_b}, we need two approximation results in the spirit of Lemmas \ref{lemma_lip1} and \ref{lemma_lip2}. The first allows a discretization for $s$, the second (which is only needed for Lemma \ref{lem_th2_large_b}) allows to pass to a subset of the set of values of $N$ that actually need to be considered.

\begin{lem} \label{lemma_lip3}
For any $s,s' \in [0,1/2]$ and any $N \geq 1$ we have
$$
\left| V_N(s) - V_N(s') \right| \leq V_N \left( \left| s- s' \right| \right) + 2 \sqrt{V_N(s') V_N(|s-s'|)}.
$$
\end{lem}

\begin{proof}
    Recall the definition (\ref{V_def}) of $V_N(s)$. Assume that $0<s^\prime<s <1/2$. We have
    \begin{eqnarray}
V_N(s) & = & \int_0^1 \left( S_N(x,s) - Ns \right)^2 dx \nonumber \\
& = & \int_0^1 \left(  \sum_{n=1}^N \mathbf{1}_{[x,x+s]}(X_n) - Ns \right)^2 dx\nonumber \\
& = & \int_0^1 \left( \sum_{n=1}^N \mathbf{1}_{[x,x+s^\prime]}(X_n) - Ns^\prime+ \sum_{n=1}^N \mathbf{1}_{[x+s^\prime,x+s]}(X_n) - N(s-s^\prime) \right)^2 dx \nonumber \\
& = & V_N(s^\prime)+V_N(s -s^\prime)+\nonumber \\
& & \qquad  +2\int_0^1 \left( \sum_{n=1}^N \mathbf{1}_{[x,x+s^\prime]}(X_n) - Ns^\prime\right)\left(\sum_{n=1}^N \mathbf{1}_{[x+s^\prime,x+s]}(X_n) - N(s-s^\prime) \right) dx.  \label{prod1}
\end{eqnarray}
Applying the Cauchy-Schwarz inequality to (\ref{prod1}), we obtain 
\begin{eqnarray*}
    \left| V_N(s) - V_N(s') \right| \leq V_N \left( s-s^\prime \right) + 2 \sqrt{V_N(s') V_N(s-s')}.
\end{eqnarray*}
Similar arguments apply to the case $0<s<s^\prime<1/2$.
\end{proof}

During the proofs of Lemma \ref{lem_th2_large_a} and \ref{lem_th2_large_b}, we will switch from the continuous range of all parameters $s$ to a (sufficiently dense) discrete range $(s_v)_v$, where $v$ runs through a suitable index set. Lemma \ref{lemma_lip3} essentially allows us to control the deviation between $V_N(s)$ and $V_N(s_v)$ in terms of $V_N(|s-s_v|)$. Now the crucial point is that if we can ensure that $|s-s_v|$ is ``small'', then we already know from Lemma \ref{lem_th2_small} that $V_N(|s-s_v|)$ is (almost surely) well-behaved. In other words, we can efficiently utilize the information from Lemma \ref{lem_th2_small} for the proof of Lemma \ref{lem_th2_large_a}. In the same way, in the following step we can utilize the conclusion of Lemma \ref{lem_th2_large_a}, together with Lemma \ref{lemma_lip3}, for the proof of Lemma \ref{lem_th2_large_b}.\\

For the statement of the next lemma, we introduce the notation $V_{N+1,M-N}$ for the number variance of the points $X_{N+1}, \dots, X_M$.
\begin{lem} \label{lemma_lip4}
For any $M \geq N \geq 1$ and any $s \in [0,1/2]$ we have
$$
\left| V_N(s) - V_M(s)\right| \leq V_{N+1,M-N} (s) + 2 \sqrt{V_N(s) V_{N+1,M-N} (s)}.
$$
\end{lem}

\begin{proof}
\begin{eqnarray}
    V_M(s) & = & \int_0^1 \left( S_M(s) - Ms \right)^2 ~dx \nonumber \\
& = & \int_0^1 \left(  \sum_{n=1}^M \mathbf{1}_{[x,x+s]}(X_n) - Ms \right)^2 ~dx \nonumber \\
& = & \int_0^1 \left( \sum_{n=1}^N \mathbf{1}_{[x,x+s]}(X_n) - Ns+ \sum_{n=N+1}^M \mathbf{1}_{[x,x+s]}(X_n) - (M-N)s\right)^2 ~dx \nonumber\\
& = & V_N(s)+V_{N+1,M-N}(s)+ \nonumber \\
& & \qquad  +2\int_0^1 \left( \sum_{n=1}^N \mathbf{1}_{[x,x+s]}(X_n) - Ns\right)\left(\sum_{n=N+1}^M \mathbf{1}_{[x,x+s]}(X_n) - (M-N)s \right) ~dx . \label{prod2}
\end{eqnarray}
Applying Cauchy-Schwarz inequality to (\ref{prod2}), we obtain 
\begin{eqnarray*}
    \left| V_N(s) - V_M(s) \right| \leq  V_{N+1,M-N} (s)  + 2 \sqrt{V_N(s)  V_{N+1,M-N} (s) }.
\end{eqnarray*}
\end{proof}

\begin{proof}[Proof of Lemma \ref{lem_th2_large_a}]
Assume that $k$ is a ``large'' positive integer, and that $N \in \{2^{k-1}, \dots, 2^k\}$. We consider all values of $s$ of the form 
$$
s_v = v 2^{-4k/5}, \quad 1 \leq v \leq 3 \cdot \frac{2^{4k/5}}{k^2}.
$$
Then for every $s$ in the range $[N^{-3/4}, (\log N)^{-2}]$ there exists a $v$ in the range specified above such that
\begin{equation} \label{s_approx}
|s - s_v| \leq 2^{-4k/5}.
\end{equation}
To see this, note that
$$
\frac{1}{(\log N)^2} \leq \frac{1}{(\log 2^{k-1})^2} \leq \frac{1}{(\log 2)^2} \frac{1}{(k-1)^2} \leq \frac{3}{k^2}
$$
for sufficiently large $k$. \\

For given $k$ and for $N$ and $v$ in the ranges specified above we define sets 
$$
A_{k,v,N} = \left\{ \left| N \int_0^1 \left( B(x + s_v) - B(x) \right)^2 dx - s_v(1-s_v) N \right| > s_v^{11/10} N \right\}, 
$$
where $B$ is the (periodized) Brownian Bridge process which comes from the KMT-approximation, as described in Section \ref{sec_KMT}. As already noted the KMT approximation constructs different Brownian Bridges for different values of $N$, but since we are only interested in distributions and exceptional probabilities (for a given value of $N$), we will suppress this dependence in the sequel and simply write $B$ rather than $B_N$. By Lemma \ref{lem_H_tail} we have
\begin{equation} \label{size_of_prob}
\p (A_{k,v,N}) \leq 2 e^{-\frac{s_v^{22/10}}{4 s_v^3}} = 2 e^{-\frac{1}{4 s_v^{8/10}}} \leq 2 e^{-\frac{(k^2/3)^{8/10}}{4}} = 2 e^{-\frac{k^{8/5}}{10}}.
\end{equation}
Note that Lemma \ref{lem_H_tail} is indeed applicable since $s_v^{11/10} \leq s_v\left( \frac{3}{k^2} \right)^{1/10} \leq s_v/100$ is satisfied for sufficiently large $k$. From \eqref{size_of_prob} one can easily see that
\begin{equation} \label{poly_sum}
\sum_{k=1}^\infty \sum_{N = 2^{k-1}}^{2^k} \sum_{v=1}^{3 \cdot \frac{2^{4k/5}}{k^2}} \p (A_{k,v,N}) < + \infty, 
\end{equation}
so that accordingly by the Borel-Cantelli lemma with probability one only finitely many events $A_{k,v,N}$ occur. Now assume that we are in the almost sure ``generic'' situation, where only finitely many events $A_{k,v,N}$ have occurred. Assume that $N \in \{2^{k-1}, \dots, 2^k\}$ for $k$ so large that no event $A_{k,v,N}$ has occurred (for any admissible value of $v$). Let $s \in [N^{-3/4}, (\log N)^{-2}]$ be arbitrary. Since $A_{k,v,N}$ has not occurred we can assume that
$$
\left| N \int_0^1 \left( B(x + s_v) - B(x) \right)^2 dx - s_v N \right| \leq s_v^{11/10} N,
$$
and as a consequence also
$$
\int_0^1 \left( B(x + s_v) - B(x) \right)^2 dx \leq 2 s_v
$$
for the Brownian Bridge from the KMT approximation. Applying \eqref{KMT_error_2} and the Cauchy--Schwarz inequality we have, for some absolute constant $c$, 
\begin{eqnarray}
&& \left| V_N(s_v) - N s_v (1-s_v) \right| \label{V_diff} \\
& \leq & s_v^{11/10} N + \left| V_N(s_v) - N \int_0^1 \left( B(x+s_v) - B(x) \right)^2 dx \right| \nonumber\\
&\leq & s_v^{11/10} N + 4  c \log N \sqrt{N} \int_0^1 \left| B(x+s_v) - B(x) \right| dx + 4c^2 (\log N)^2 \nonumber\\
& \leq & s_v^{11/10} N + 4  c \log N \sqrt{N} \left(\int_0^1 \left( B(x+s_v) - B(x) \right)^2 dx \right)^{1/2} + 4 c^2(\log N)^2 \nonumber\\
& \leq & s_v^{11/10} N + 4  c \log N \sqrt{N} \sqrt{ 2s_v} + 4 c^2 (\log N)^2 \nonumber\\
& = & \mathcal{O} \left( \frac{s_v N}{(\log N)^{1/5}} \right) \nonumber
\end{eqnarray}
by our assumption on the possible size of $s_v$, where the implied constant does not depend on $s_v$ or $N$. 
This clearly also implies
\begin{equation}
\left| V_N(s_v) - N s_v \right| =  \mathcal{O} \left( \frac{s_v N}{(\log N)^{1/5}} \right), \label{continue_here}
\end{equation}
since the $N s_v^2$ term from \eqref{V_diff} can be integrated in the $\mathcal{O} \left( \frac{s_v N}{(\log N)^{1/5}} \right)$ error in the given range of $s_v$. We acknowledge the fact that the KMT approximation in \eqref{KMT_error_2} has an exceptional probability of order $N^{-100}$ (which in terms of $k$ is order roughly $2^{-100k}$), but since there are at most $2^{2k}$ many values of $N$ and $v$ to be considered, the Borel--Cantelli lemma easily shows that the transition from $V_N$ to $B$ in the calculation above is almost surely admissible. From \eqref{continue_here} we can in particular conclude that
$$
V_N(s_v) \leq 2 s_v N
$$
for all $v$ in the admissible range. Now we move from $V_N(s_v)$ to $V_N(s)$. By Lemma \ref{lemma_lip3} we have 
\begin{equation} \label{from_this_line}
\left| V_N(s) - V_N(s_v) \right| \leq V_N(|s - s_v|) + 2 \sqrt{V_N(s_v) V_N(|s-s_v|)}.
\end{equation}
Since by \eqref{s_approx} we have $|s - s_v| \leq 2^{-4k/5} \leq N^{-3/4}$ (for sufficiently large $k$), by Lemma \ref{lem_th2_small} we
have 
$$
V_N(|s-s_v|) \leq 2 N |s-s_v| \leq 2 N 2^{-4k/5},
$$
almost surely for sufficiently large $k$, and thus 
$$
\left| V_N(s) - V_N(s_v) \right| \leq 2 N |s-s_v|  + 2 \sqrt{2 N s_v 2 N |s-s_v|} = \mathcal{O} \left( N 2^{-4k/5} + N \sqrt{s_v} 2^{-2k/5} \right)
$$
almost surely for sufficiently large $k$. Thus, taking \eqref{continue_here} into account, and using that $2^k$ is of order $N$, together with the assumption that $s \geq N^{-3/4}$, we have
\begin{eqnarray}
\left|V_N(s) - Ns \right| & \leq & \left| V_N(s) - V_N(s_v) \right| + \left| V_N(s_v) - N s_v \right| + N|s - s_v| \nonumber\\
& = & \mathcal{O} \left(N \sqrt{s_v} 2^{-2k/5}+ \frac{s_v N}{(\log N)^{1/5}} + N 2^{-4k/5}  \right) \nonumber\\
& = & \mathcal{O} \left( N \sqrt{s} 2^{-2k/5}+\frac{s N}{(\log N)^{1/5}} + N 2^{-4k/5}  \right) \nonumber\\
& = & \mathcal{O} \left( N s N^{3/8} 2^{-2k/5} +\frac{s N}{(\log N)^{1/5}} + N 2^{-4k/5} \right) \nonumber\\
& = & \mathcal{O} \left( N s N^{-1/40}+\frac{s N}{(\log N)^{1/5}} + N s N^{-1/20}  \right) \nonumber\\
& = & \mathcal{O} \left( \frac{s N}{(\log N)^{1/5}} \right) = o(sN) \label{to_this_line}
\end{eqnarray}
almost surely, where the implied constant is independent of $s$ and $N$. This proves Lemma \ref{lem_th2_large_a}.
\end{proof}

The proof of Lemma \ref{lem_th2_large_b} is similar to that of Lemma \ref{lem_th2_large_a}. However, the situation is more delicate. In the previous proof, the exceptional probabilities coming from the exponential inequality were of very small order (super-polynomial order), so that in \eqref{poly_sum} we could afford to take a sum over all possible values of $N$. Now in the situation of Lemma \ref{lem_th2_large_b}, the exceptional probabilities are in some regimes only of logarithmic order, and we cannot afford to simply take a union bound over all possible values of $N$. Instead, we need to sparsify the set of relevant values of the index $N$ with the help of Lemma \ref{lemma_lip4}.

\begin{proof}[Proof of Lemma \ref{lem_th2_large_b}]
Again we assume that $k$ is a ``large'' positive integer. We set $$g_k=\max\limits_{2^{k-1}\leq N\leq 2^k}g(N).$$ We first consider all values of $s$ of the form
$$
s_v = \frac{v}{k^4}, \qquad k^2 \leq v \leq \frac{2 g_k k^4}{\log k},
$$
and all values of $N$ of the form 
$$
N_j = 2^{k-1} + j \frac{2^{k-1}}{k^4}, \qquad 0 \leq j \leq k^4. 
$$
Then for all $N \in \{2^{k-1}, \dots, 2^k\}$ and for all $s \in [(\log N)^{-2}, g(N) / (\log \log N)]$ there is a $v$ in the admissible range such that 
\begin{equation} \label{svs}
|s - s_v| \leq k^{-4} = \mathcal{O} \left( (\log N)^{-4} \right). 
\end{equation}
Note here that the largest admissible value of $s_v$, for $v = \frac{2 g_k k^4}{\log k}$, is 
$$
s_v = \frac{2 g_k}{\log k} \geq \frac{g_k}{\log \log N} \geq \frac{g(N)}{\log \log N}
$$
for sufficiently large $N$, since $\log k \sim \log \log N$ as $N \to \infty$. Thus, the discretized points $s_v, ~k^2 \leq v \leq \frac{2 g_k k^4}{\log k}$ indeed cover the whole relevant range $[(\log N)^{-2}, g(N)/\log \log N]$.\\

Similar to the previous proof, we define sets 
$$
A_{k,v,j} = \left\{ \left| N_j \int_0^1 \left( B(x + s_v) - B(x) \right)^2 dx - s_v (1-s_v) N_j \right| > g_k^{1/4} s_v N_j \right\},
$$
where $B$ is the Brownian Bridge coming from the KMT approximation (and where we again suppress the dependence of $B$ on $N$ since we are only interested in distributions). Using $s_v \leq \frac{2 g_k}{\log k}$ the exponential inequality in Lemma \ref{lem_H_tail} yields
$$
\p (A_{k,v,j}) \leq 2 e^{-\frac{\sqrt{g_k} s_v^2}{4 s_v^3}} = 2 e^{-\frac{\log k}{8 \sqrt{g_k}}} = \frac{2}{k^{1/(8 \sqrt{g_k})}}
$$
for sufficiently large $k$. Note that Lemma \ref{lem_H_tail} is indeed applicable since $g_k^{1/4} s_v \leq s_v/100$ for sufficiently large $k$, due to our assumption that $g_k \to 0$ as $k \to \infty$. Again using the fact that $g_k \to 0$ by assumption we have
$$
\sum_{k=1}^\infty \sum_{j=0}^{k^4} \sum_{v=k^2}^{\frac{2 g_k k^4}{\log k}} \frac{2}{k^{1/(8 \sqrt{g_k})}} < +\infty.
$$
Thus, by the convergence Borel--Cantelli lemma, almost surely only finitely many events $A_{k,v,j}$ occur. Essentially, this is what we need to prove (with the help of the approximation Lemma \ref{lemma_lip3}) that $V_{N_j}(s) \sim s N_j$ almost surely, uniformly  throughout the relevant range of $s$, similar to the proof of Lemma \ref{lem_th2_large_a}. However, now we also need a further ingredient, which allows us to handle all possible values of $N$, and close the gaps between the sparsified index set of the $N_j$'s. To handle this, we also define sets $B_{k,v,j,r}$ that control the fluctuation of $V_N$ for $N$ between two points $N_j$ and $N_{j+1}$ from our sparsified set of indices. We write $N$ as $N_j + r$ for some suitable $r$, and use Lemma \ref{lemma_lip4} to relate $V_N$ to $V_{N_j}$ and $V_{N_j + 1,r}$. We don't require the full precision of Lemma \ref{lem_H_tail} here (which comes with a bound on the maximal value of the parameter $x$ in the large deviations estimate, which would actually prevent us from using the lemma in this situation).  More precisely, we define 
$$
B_{k,v,j,r} =  \left\{ r  \max_{x \in [0,1]} B(x)^2 > s_v^{3/2} 2^k \right\},
$$
where we use the Brownian motion $B$ which the KMT approximation constructs to approximate the empirical distribution of the points $X_{N_j+1}, \dots, X_{N_j + r}$. Here $v$ and $j$ run through the parameter ranges as specified above, and $r$ runs through the range $1 \leq r < 2^{k-1}/k^4$. 
The probability of $B_{k,v,j,r}$ can be easily estimated from the formula for the distribution of the maximum of the Brownian bridge (which is known as ``Kolmogorov's distribution'', and is at the basis of the Kolmogorov--Smirnov test). Writing $u = \frac{s_v^{3/2} 2^k}{2r}$ we obtain
\begin{eqnarray}
\p (B_{k,v,j,r}) & \leq & \p \left( \max_{x \in [0,1]} B(x)^2 > \frac{s_v^{3/2} 2^k}{2r} \right) \label{bridge_max}\\
& = & 2 \sum_{m=1}^\infty (-1)^{m-1} e^{-2m^2 u}. \nonumber
\end{eqnarray}
This is an alternating series whose terms are decreasing in absolute value, and since such a series is dominated by its first term we have
\begin{eqnarray}
\p (B_{k,v,j,r}) & \leq & 2 e^{-2 u} = 2 e^{-\frac{s_v^{3/2} 2^{k}}{r}} \leq 2 e^{-s_v^{3/2} k^4} \leq 2 e^{-k}. \label{bridge_max_2}
\end{eqnarray}
Accordingly, summing over $k$ and all permissible parameters for $v,j,r$, we have
\begin{eqnarray*}
\sum_{k=1}^\infty \sum_{k^2 \leq v \leq 2 g_k k^4 / \log k} \sum_{j=0}^{k^4} \sum_{1 \leq r < 2^{k-1}/k^4} \p (B_{k,v,j,r}) & \leq & \sum_{k=1}^\infty 2^k k^4 2 e^{-k} < + \infty.
\end{eqnarray*}
Thus, according to the Borel-Cantelli lemma, with probability one only finitely many events $B_{k,v,j,r}$ occur, which will allow us to control the (almost sure) size of $V_{N_j+1,r}$. More precisely, if $B_{k,v,j,r}$ has not occurred, then by \eqref{KMT_error_2}  we have
\begin{eqnarray}
& & V_{N_j+1,r}(s_v) \nonumber\\
& \leq & 4r  \max_{x \in [0,1]} B(x)^2  + 8c \log r \sqrt{r}  \max_{x \in [0,1]} |B(x)| + 4 c^2 (\log r)^2 \nonumber\\
& \leq & 2 s_v^{3/2} 2^k +  4\sqrt{2}c \log r  s_v^{3/4} 2^{k/2} + 4 c^2 (\log r)^2 \nonumber\\
& \leq & \mathcal{O}(s_v^{3/2} 2^k),  \label{bound_for_jrs}
\end{eqnarray}
uniformly in $s$ and $N$. As in the proof of Lemma \ref{lem_th2_large_a}, we acknowledge the fact that the KMT approximation is only valid up to a certain exceptional probability. When applying the KMT approximation for very small value of $r$, these exceptional probabilities (of order $r^{-100}$) are not necessarily small. However, this can be easily fixed. For values of $r$ which are reasonably large, say $r \geq 2^{k/4}$, we apply the KMT approximation as indicated above; this introduces exceptional probabilities of order at most $r^{-100}  \leq 2^{-25 k}$, which are easily incorporated into a convergent series and taken care of by the Borel--Cantelli lemma. For very small value of $r$, say $r \leq 2^{k/4}$, we have the trivial upper bound $V_{N_j+1,r} (s_v) \leq r^2 s_v$, which is easily seen to also imply \eqref{bound_for_jrs}.\\

Now assume that we are in the ``typical'' situation where we have a $k$ (for $k$ large enough) such that no event $A_{k,v,j}$ or $B_{k,v,j,r}$ has occurred, for any of the admissible values of $v,j,r$. Let $N \in \{2^{k-1}, \dots, 2^k\}$ be given, and let $s$ be given with $(\log N)^{-2} \leq s \leq g(N) / \log \log N$. There are $j$ and $r$ (within the range of admissible values) such that $N = N_j + r$, and by \eqref{svs} there is $s_v$ for an admissible value of $v$ such that $|s - s_v| \leq k^{-4}$. Using $N = N_j + r$, by Lemma \ref{lemma_lip4} we 
\begin{equation} \label{V_N_and_V_M}
\left| V_N(s_v) - V_{N_j}(s_v)\right| \leq V_{N_j+1,r} (s_v) + 2\sqrt{V_{N_j}(s_v) V_{N_j+1,r} (s_v)}.
\end{equation}
Applying \eqref{KMT_error_2} and the Cauchy--Schwarz inequality,  and using the assumption that $A_{k,v,j}$ has not occurred, we have
\begin{eqnarray}
& & \left| V_{N_j}(s_v) - N_j s_v (1-s_v) \right| \label{V_diff_2} \\
& \leq & g_k^{1/4} s_v N_j + \left| V_{N_j}(s_v) - N_j \int_0^1 \left( B(x+s_v) - B(x) \right)^2 dx \right| \nonumber\\
&\leq & g_k^{1/4} s_v N_j + 4  c \log N_j \sqrt{N_j} \int_0^1 \left| B(x+s_v) - B(x) \right| dx + 4c^2 (\log N_j)^2 \nonumber\\
& \leq & g_k^{1/4} s_v N_j + 4  c \log N_j \sqrt{N_j} \left(\int_0^1 \left( B(x+s_v) - B(x) \right)^2 dx \right)^{1/2} + 4 c^2 (\log N_j)^2 \nonumber\\
& \leq & g_k^{1/4} s_v N_j + 4  c \log N \sqrt{N} g_k^{1/8} \sqrt{s_v} + 4 c^2 (\log N_j)^2 \nonumber\\
& = & o \left(s N \right), \nonumber
\end{eqnarray}
independent of the particular value of $s$ and $N$ (that is, of $v$ and $j$). This also implies
\begin{equation} \label{incombination}
\left| V_{N_j}(s_v) - N_j s_v\right| = o \left(s N \right),
\end{equation}
since the $N_j s_v^2$ term from \eqref{V_diff_2} can be incorporated into the error term $o(sN)$. In particular, $V_{N_j}(s_v) \leq 2 s_v N_j$ for sufficiently large $k$. This also implies Using this is \eqref{V_N_and_V_M}, we get
\begin{equation} \label{plug_into}
\left| V_N(s_v) - V_{N_j}(s_v)\right| \leq V_{N_j+1,r} (s_v) + 2\sqrt{2 s_v N_j V_{N_j+1,r} (s_v)}. 
\end{equation}
Using the assumption that $B_{k,v,j,r}$ has not occurred, we can use the estimate \eqref{bound_for_jrs} from above. Plugging this into \eqref{plug_into} yields
$$
\left| V_N(s_v) - V_{N_j}(s_v)\right| = \mathcal{O}(s^{3/2} N)  + \mathcal{O} \left(s^{5/4} N \right).
$$
In combination with \eqref{incombination} this gives
\begin{equation}  \label{vn_conv}
\left| V_N(s_v) - N s_v \right|\leq |V_N(s_v)-V_{N_j}(s_v)|+|V_{N_j}(s_v)-N_js_v|+|N_js_v-Ns_v|  = o(sN), 
\end{equation}
uniformly in $N$ and $s$. Now it remains to switch from $s_v$ to $s$. This is done using Lemma \ref{lemma_lip3}, and since the argument is completely similar to the one given in lines \eqref{from_this_line}--\eqref{to_this_line}, we omit the details and just note that the discrete estimate \eqref{vn_conv} translates to the continuous estimate 
$$
\left| V_N(s) - N s \right|  = o(sN).
$$
This holds almost surely, uniformly throughout the relevant range of $s$, which proves the assertion of the lemma.
\end{proof}

Lemmas \ref{lem_th2_small}, \ref{lem_th2_large_a} and \ref{lem_th2_large_b} together cover all ranges for the convergence part of Theorem \ref{th_2}.

\section{Almost sure behavior of $V_N(s)$. Proof of Theorem \ref{th_2}, divergence part.} \label{sec_th3}

In this section we will prove the divergence part of Theorem \ref{th_2}. Accordingly, let $\varepsilon>0$ be fixed throughout this section. We will establish that there is an $\eta>0$ such that almost surely there exist infinitely many $N$ such that for $s = \varepsilon / \log \log N$ we have
$$
\left| \frac{V_N(s)}{Ns} - 1\right| > \eta. 
$$
almost surely.\\

Recall that $\varepsilon > 0$ is assumed to be fixed. We set 
$$
\eta = \min \left\{ \frac{1}{5000}, \frac{\varepsilon}{30} \right\}.
$$
Furthermore, we choose an integer $\theta \geq 18$ which is so large that
\begin{equation} \label{choice_theta}
\left( \sqrt{1+2 \eta} - \theta^{-1/4} \right)^2 \geq 1+\eta.
\end{equation}
Let
$$
N_j = \theta^j, \qquad j \geq 1.
$$
Using the notation introduced before the statement of Lemma \ref{lemma_lip4}, we write $V_{N_{j-1}+1,N_j - N_{j-1}}$ for the number variance of the points $X_{N_{j-1}+1},\dots, X_{N_{j}}$. We write 
$$
s_{j} = \frac{\varepsilon}{\log \log  N_{j}},
$$ 
and define sets 
$$
A_j = \left\{  \left| V_{N_{j-1}+1,N_j - N_{j-1}}(s_{j}) - s_j  (N_j - N_{j-1})\right| > 2 \eta s_j (N_j - N_{j-1}) \right\}.
$$
Using again the equations in \eqref{KMT_main}--\eqref{KMT_error_2} to switch from the number variance to the Brownian bridge, we get
\begin{eqnarray*}
    \mathbb{P}(A_j) & \geq & \mathbb{P} \left\{  \left| \int_0^1 \left( B(x+s_j) - B(x)\right)^2 dx - s_j  \right| > 4 \eta s_j \right\} \\
    & & \qquad - \mathbb{P} \left\{ \frac{4 c \log (N_j - N_{j-1})}{\sqrt{N_j - N_{j-1}}} \int_0^1 |B(x+s_j) - B(x)| dx > \eta s_j \right\}
\end{eqnarray*}
(where again we simply ignore the $N^{-100}$ probability from the KMT approximation) for all $j$ which are so large that 
$$
\frac{4 c^2(\log (N_j - N_{j-1}))^2}{N_j - N_{j-1}} \leq \eta s_j.
$$
We have 
\begin{eqnarray*}
& & \mathbb{P} \left\{ \frac{4 c \log (N_j - N_{j-1})}{\sqrt{N_j - N_{j-1}}} \int_0^1 |B(x+s_j) - B(x)| dx > \eta s_j \right\} \\
& \leq & \mathbb{P} \left\{ 2 \max_{x \in [0,1]} |B(x)| > \frac{\eta s_j \sqrt{N_j - N_{j-1}}}{4 c \log (N_j - N_{j-1})} \right\},
\end{eqnarray*}
and a calculation similar to the one in \eqref{bridge_max}--\eqref{bridge_max_2} shows that these probabilities are extremely small (they yield a convergent series when summing over $j$). On the other hand, for sufficiently large $j$ we have
\begin{eqnarray*}
&& \mathbb{P} \left\{  \left| \int_0^1 \left( B(x+s_j) - B(x)\right)^2 dx - s_j  \right| > 4 \eta s_j \right\} \\
& \geq & \mathbb{P} \left\{  \left| \int_0^1 \left( B(x+s_j) - B(x)\right)^2 dx - s_j(1-s_j)  \right| > 5 \eta s_j \right\} \\
& = & \mathbb{P} \left\{  \left|H - s_j(1-s_j)  \right| > 5 \eta s_j \right\}.
\end{eqnarray*}
Now we apply Lemma \ref{lemma_lower_tail} to get a lower bound for this probability. Note that by our choice of $\eta$ we have $5 \eta \leq 1/1000$, so Lemma \ref{lemma_lower_tail} is indeed applicable (assuming also that $j$ is so large that $s_j \leq 1/100$). Note that from our choice of $\eta$ we have
$$
\frac{\eta^2}{\varepsilon^2} = \frac{\left(\min \left\{ \frac{1}{5000}, \frac{\varepsilon}{30}\right\} \right)^2}{\varepsilon^2} \leq \frac{1}{900}.
$$
Thus from Lemma \ref{lemma_lower_tail} we get, for all sufficiently large $j$, that
\begin{eqnarray*}
\mathbb{P} \left\{  \left|H - s_j(1-s_j)  \right| > 5 \eta s_j \right\} & \geq &\min\left\{ \frac{ 5^4 \eta^4}{ s_j^2} ,\frac{1}{4}\right\}e^{-\frac{17\cdot 25 \eta^2}{s_j}} \\
& \geq & \tfrac{1}{4}e^{-\frac{425 \eta^2 \log \log \theta^j}{\varepsilon^2}} \\
& = & \tfrac{1}{4}e^{-\frac{425 \log \log \theta^j}{900}} \\
& = & \frac{1}{4(j \log \theta)^{\frac{17}{36}}},
\end{eqnarray*}
which leads to a divergent series when summing over $j$. Overall this shows that
$$
\sum_{j=1}^\infty \mathbb{P} (A_j) = + \infty.
$$
By construction, the events $(A_j)_{j \geq 1}$ are independent, since each $A_j$ only depends on the random variables $X_{N_{j-1} + 1}, \dots, X_{N_j}$. Thus, by the divergence Borel--Cantelli lemma, almost surely there exist infinitely many values of $j$ such that the event $A_j$ occurs. Now assume that $A_j$ has occurred. Note that $A_j$ makes a statement about $V_{N_{j-1}+1,N_j - N_{j-1}}$, i.e.\ about the number variance of $X_{N_{j-1}+1},\dots,X_{N_j}$, but we are actually interested in $V_{N_j}$, i.e.\ the number variance of the full point set $X_1, \dots, X_{N_j}$. Note that $V_{N_{j-1}+1,N_{j} - N_{j-1}}$ interacts with $V_{N_{j-1}}$ to give $V_{N_j}$. We will now create a dichotomy: We have shown that infinitely often the ``increment'' $V_{N_{j-1}+1,N_{j} - N_{j-1}}$ is large. If $V_{N_{j-1}}$ is small, then we will be able to conclude that $V_{N_j}$ is also large (almost as large as the increment). In contrast, if $V_{N_{j-1}}$ is sufficiently large to compensate for $V_{N_{j-1}+1,N_{j} - N_{j-1}}$ and make $V_{N_j}$ small, then we have also found an index where the number variance is large: not $N_j$, but $N_{j-1}$. In either case, the threshold is exceeded infinitely often.\\

Recall equation \eqref{prod2} from the proof of Lemma \ref{lemma_lip4}. Applying Cauchy--Schwarz as in that proof, but only after bringing $V_N(s)$ and the integral to the left-hand side, and re-writing the resulting inequality in our present setup (where we have $N = N_{j-1}$, $M = N_j$, $M-N = N_j - N_{j-1}$, and $s= s_j$) we have
\begin{eqnarray} 
V_{N_j}(s_j) & \geq & V_{N_{j-1}}(s_j) + V_{N_{j-1}+1,N_j - N_{j-1}} (s_j) - 2 \sqrt{V_{N_{j-1}}(s_j) V_{N_{j-1}+1,N_j - N_{j-1}} (s_j)} \nonumber\\
& = & \left( \sqrt{V_{N_{j-1}+1,N_j - N_{j-1}} (s_j)} -  \sqrt{V_{N_{j-1}}(s_j)} \right)^2.\label{all_sorts_of_variances}
\end{eqnarray}
Note how \eqref{all_sorts_of_variances} affirms the heuristic from above: if we know that $V_{N_{j-1}+1,N_j - N_{j-1}} (s_j)$ is large, then necessarily at least one of the other two terms in the formula must be large as well. To make this precise, we define sets
\begin{equation} \label{avoid_this}
B_j := \left\{ V_{N_{j-1}} (s_j) < \frac{s_j (N_j - N_{j-1})}{\sqrt{\theta}} \right\}.
\end{equation}
Then, writing $B_j^c$ for the complement of $B_j$, we have
\begin{eqnarray*}
& & \p \left( (A_j \cap B_j) \cup B_j^c \quad \text{occurs for infinitely many $j$} \right)\\
& \geq & \p ( (A_j \cap B_j) \cup (A_j \cap B_j^c) \quad \text{occurs for infinitely many $j$}) \\
& = & \p ( A_j \quad \text{occurs for infinitely many $j$}) \\
& = & 1.
\end{eqnarray*}

If $A_j$ and $B_j$ have jointly occurred infinitely often, then for infinitely many $j$ we have
$$
V_{N_{j-1}+1,N_j - N_{j-1}} (s_j) \geq (1 + 2 \eta ) s_j (N_j - N_{j-1})
$$
by our definition of $A_j$, and by \eqref{all_sorts_of_variances} we have
\begin{eqnarray*}
V_{N_j}(s_j) & \geq & \left( \sqrt{1+2 \eta} - \theta^{-1/4} \right)^2  s_j (N_j - N_{j-1}) \\
& \geq & \left( 1 + \eta \right)  s_j N_j 
\end{eqnarray*}
for infinitely many $j$ as a consequence of our choice of $\theta$ in \eqref{choice_theta}. Thus the conclusion of the divergence part of Theorem \ref{th_2} is satisfied in this case. 

It remains to show that the conclusion of the divergence part of Theorem \ref{th_2} is satisfied in the case when $B_j^c$ has occurred for infinitely many $j$. In other words, we may assume that for infinitely many $j$ we have
$$
V_{N_{j-1}} (s_j) \geq \frac{s_j (N_j - N_{j-1})}{\sqrt{\theta}},
$$
which can be re-written as
\begin{equation} \label{B_c_has_happened}
V_{N_{j-1}} (s_j) \geq s_j N_{j-1} \frac{\theta-1}{\sqrt{\theta}}.
\end{equation}
This is very similar to what we need to establish for the divergence part of Theorem \ref{th_2}, only that for the conclusion of the theorem we would need a statement about $V_{N_{j-1}} (s_{j-1})$, not about $V_{N_{j-1}}(s_j)$. Thus we need to change the argument of $V_{N_{j-1}}$ from $s_j$ to $s_{j-1}$. It is an easy exercise (using derivatives) to show that
\begin{equation}\label{diff_sj}
s_j \leq s_{j-1} \leq s_j + \frac{\varepsilon \theta}{(\log (\theta^{j-1}))(\log \log (\theta^{j-1}))^2}.
\end{equation}
By Lemma \ref{lemma_lip3} we have
\begin{eqnarray} 
& & \left| V_{N_{j-1}}(s_j) - V_{N_{j-1}}(s_{j-1}) \right| \nonumber\\
& \leq & V_{N_{j-1}} \left( \left| s_j- s_{j-1} \right| \right) + 2 \sqrt{V_{N_{j-1}}(s_{j-1}) V_{N_{j-1}}(|s_j-s_{j-1}|)}. \label{s_j_s_j-1}
\end{eqnarray}
Now by \eqref{diff_sj} we have
$$
\left| s_{j-1} - s_j \right| \leq \frac{\varepsilon \theta}{(\log (\theta^{j-1}))(\log \log (\theta^{j-1}))^2} \leq \frac{1}{\log N_{j-1}}
$$
for sufficiently large $j$, so that the number variance of $V_{N_{j-1}} \left( \left| s_j- s_{j-1} \right| \right)$ easily falls within the convergence part of Theorem \ref{th_2}. Accordingly, almost surely we have 
$$
V_{N_{j-1}} \left( \left| s_j- s_{j-1} \right| \right) \leq 2 N_{j-1} |s_j - s_{j-1}| \leq \frac{2 N_{j-1}}{\log N_{j-1}}
$$
for all sufficiently large $j$. Thus if \eqref{B_c_has_happened} is true, then by \eqref{s_j_s_j-1} we have
\begin{eqnarray}
V_{N_{j-1}}(s_{j-1}) & \geq & V_{N_{j-1}}(s_j) - \frac{2 N_{j-1}}{\log N_{j-1}} - 2 \sqrt{\frac{2 V_{N_{j-1}}(s_{j-1}) N_{j-1}}{\log N_{j-1}}} \nonumber\\
& \geq & s_j N_{j-1} \frac{\theta-1}{\sqrt{\theta}} - \frac{2 N_{j-1}}{\log N_{j-1}} - 2 \sqrt{\frac{2 V_{N_{j-1}}(s_{j-1}) N_{j-1}}{\log N_{j-1}}} \label{plug_in_here}
\end{eqnarray}
for infinitely many $j$ (up to a set of probability zero). If we assume
\begin{equation} \label{assume_this}
V_{N_{j-1}}(s_{j-1}) \leq s_j N_{j-1} \frac{\theta-1}{\sqrt{\theta}},
\end{equation}
then we can plug this into \eqref{plug_in_here} and get
\begin{eqnarray}
V_{N_{j-1}}(s_{j-1}) & \geq & s_j N_{j-1} \frac{\theta-1}{\sqrt{\theta}} - \frac{2 N_{j-1}}{\log N_{j-1}} - 2 N_{j-1}\sqrt{\frac{2 s_j (\theta-1)} {\sqrt{\theta}\log N_{j-1}}} \nonumber\\
& \geq &  s_{j-1} N_{j-1} \frac{\theta-1}{2\sqrt{\theta}} \label{last_two}
\end{eqnarray}
(since the last two terms on the right-hand side of \eqref{last_two} are of smaller asymptotic order than the first term on the right-hand side, and since it is easily seen that $s_{j-1} / s_j \to 1$ as $j \to \infty$). On the other hand, if \eqref{assume_this} fails to be true, then we directly obtain
$$
V_{N_{j-1}}(s_{j-1}) \geq s_j N_{j-1} \frac{\theta-1}{\sqrt{\theta}} \geq s_{j-1} N_{j-1} \frac{\theta-1}{2\sqrt{\theta}}
$$
since $s_j \sim s_{j-1}$. Thus in both case we have
$$
V_{N_{j-1}}(s_{j-1}) \geq s_{j-1} N_{j-1} \frac{\theta-1}{2\sqrt{\theta}}
$$
for infinitely many values of $j$. Since $\theta \geq 18$ by assumption, we have $\frac{\theta-1}{2 \sqrt{\theta}} \geq 2 \geq 1 + \eta$ (since $\eta$ was chosen smaller than $1$). Thus if $B_j^c$ has occured for infinitely many $j$, then we have 
$$
V_{N_{j-1}}(s_{j-1}) \geq (1+\eta) s_{j-1} N_{j-1}
$$
for infinitely many values of $j$ (up to a set of probability zero). This proves the remaining case of the divergence part of Theorem \ref{th_2}.

\section*{Acknowledgments}

CA was supported by the Austrian Science Fund (FWF), projects 10.55776/I5554, 10.55776/P34763, 10.55776/P35322, and 10.55776/PAT5120424. MM was supported by the Austrian Science Fund (FWF), project 10.55776/DOC183. SH was supported by the Austrian Science Fund (FWF), project 10.55776/P35520. CA and MM were both supported by Austria's Agency for Education and Internationalization (OeAD), project BG 07/2025 ``Linear Programming Methods for Energy and Designs''.

\end{document}